\documentclass{article}

     \usepackage[preprint,nonatbib]{neurips_2020}

\usepackage[utf8]{inputenc} % allow utf-8 input
\usepackage[T1]{fontenc}    % use 8-bit T1 fonts
\usepackage{hyperref}       % hyperlinks
\usepackage{url}            % simple URL typesetting
\usepackage{booktabs}       % professional-quality tables
\usepackage{amsfonts}       % blackboard math symbols
\usepackage{nicefrac}       % compact symbols for 1/2, etc.
\usepackage{microtype}      % microtypography
\usepackage{amsmath}
\usepackage{amsthm}
\usepackage{amssymb}
\usepackage{mathrsfs}
\usepackage{bbm}
\usepackage{todonotes}
\newtheorem{theorem}{Theorem}
\newtheorem{lemma}{Lemma}

\newtheorem{definition}{Definition}
\newtheorem{proposition}{Proposition}

\newtheorem{remark}{Remark}

\allowdisplaybreaks
\usepackage{comment}

\usepackage{algorithm}
\usepackage{algpseudocode}

\usepackage{float}
\makeatletter
\newenvironment{breakablealgorithm}
  {% \begin{breakablealgorithm}
   \begin{center}
     \refstepcounter{algorithm}% New algorithm
     \hrule height.8pt depth0pt \kern2pt% \@fs@pre for \@fs@ruled
     \renewcommand{\caption}[2][\relax]{% Make a new \caption
       {\raggedright\textbf{\ALG@name~\thealgorithm} ##2\par}%
       \ifx\relax##1\relax % #1 is \relax
         \addcontentsline{loa}{algorithm}{\protect\numberline{\thealgorithm}##2}%
       \else % #1 is not \relax
         \addcontentsline{loa}{algorithm}{\protect\numberline{\thealgorithm}##1}%
       \fi
       \kern2pt\hrule\kern2pt
     }
  }{% \end{breakablealgorithm}
     \kern2pt\hrule\relax% \@fs@post for \@fs@ruled
   \end{center}
  }
\makeatother
\algnewcommand\algorithmicforeach{\textbf{for each}}
\algdef{S}[FOR]{ForEach}[1]{\algorithmicforeach\ #1\ \algorithmicdo}

\title{Estimation of Monotone Multi-Index Models}

\author{%
   David Gamarnik \thanks{http://web.mit.edu/gamarnik/www/home.html} \\
  Sloan School of Management\\
  Massachusetts Institute of Technology \\
  Cambridge, MA 02139\\
  \texttt{gamarnik@mit.edu} \\
  \And
   Julia Gaudio\thanks{http://web.mit.edu/jgaudio/www/index.html} \\
  Department of Mathematics\\
  Massachusetts Institute of Technology\\
  Cambridge, MA 02139 \\
  \texttt{jgaudio@mit.edu} \\
}
\begin{document}

\maketitle

\begin{abstract}
In a multi-index model with $k$ index vectors, the input variables are transformed by taking inner products with the index vectors. A transfer function $f: \mathbb{R}^k \to \mathbb{R}$ is applied to these inner products to generate the output. Thus, multi-index models are a generalization of linear models. In this paper, we consider \emph{monotone} multi-index models. Namely, the transfer function is assumed to be coordinate-wise monotone. The monotone multi-index model therefore generalizes both linear regression and isotonic regression, which is the estimation of a coordinate-wise monotone function. We consider the case of nonnegative index vectors. We provide an algorithm based on integer programming for the estimation of monotone multi-index models, and provide guarantees on the $L_2$ loss of the estimated function relative to the ground truth.
\end{abstract}

\section{Introduction}
Let $\beta$ be a $d \times k$ matrix, and let $f$ be a function from $\mathbb{R}^k$ to $\mathbb{R}$. The model $\mathbb{E}[Y|X] = f(\beta^T X)$ is known as a \emph{multi-index model}. The columns of $\beta$ are referred to as the \emph{index vectors} and $f$ is called a \emph{transfer function}. Therefore, multi-index models generalize linear models. Typically, $f$ is assumed to lie in a particular function class. In this paper, we assume that $f$ is coordinate-wise monotone and satisfies a mild Lipschitz condition. We treat the case where the components of $X$ are i.i.d., and the entries of $\beta$ are nonnegative.
 %We treat $d$ as large, potentially much larger than the sample size $n$. On the other hand, the number of indices $k$ is treated as constant. 

Supposing that the index vectors were known, the estimation of the function $f$ would reduce to \emph{isotonic regression}, which is the problem of estimating an unknown coordinate-wise monotone function. Monotone multi-index models (MMI) thereby additionally generalize isotonic regression. The setting where the transfer function is known is called the \emph{Generalized Index Model}, a widely applicable statistical model \cite{Dobson2018}. We are therefore considering a much more challenging model. 

We consider a high-dimensional setting, namely the dimension $d$ is possibly much larger than the sample size, $n$. We solve a sparse high-dimensional model; we assume that the number of index vectors (columns of $\beta$) is constant, and that $\beta$ has a constant number $s$ of nonzero rows. Finally, we assume that $\beta$ is a nonnegative matrix, which is natural in many applications. For example, consider the following finance application. Suppose there are $k$ future time periods, and $d$ products. Let $\beta(i,j)$ be the predicted monetary value of owning one unit of product $i$ at a time $j$. Given a vector $x$ of product quantities, the value $\beta^T x$ is a $k$-dimensional vector indicating the value of the products over the $k$ time periods. Let $f$ be a time-discounted measure of the overall value of the goods. Taking the example further, row sparsity would model an inventory restriction where one can store only $s$ distinct types of goods.
%the case where a buyer wishes to buy exactly $s$ goods at the end of the $k$ periods, so that possessing any other good has zero value.

Work on index models has largely focused on the single-index model ($k=1$) (e.g. \cite{Hardle1993}, \cite{Horowitz1996}, \cite{Ichimura1991},\cite{Kalai2009}, \cite{Kakade2011}). In particular, \cite{Kalai2009} provides the first provably efficient estimation algorithm for estimation of single-index models under monotonicity and Lipschitz assumptions. This work is further improved by \cite{Kakade2011}. %As part of their algorithm, they fit a monotone Lipschitz function. 
%In the multi-index setting that we consider ($k \geq 2$), we instead fit a coordinate-wise monotone function. 
To our knowledge, our paper is the first work done on estimation of multi-index models under the monotone Lipschitz model. 

\subsection{Notation}
Let $x$ be a vector in $\mathbb{R}^d$. The vector $p$-norm $\Vert x \Vert_p$ is defined as $\Vert x \Vert_p^p \triangleq \sum_{i=1}^d x_i^p.$
The $\infty$-norm is defined as $\Vert x \Vert_{\infty} \triangleq \max_{i \in [d]} |x_i|$. Let $\mathcal{M}_{d,k}(r)$ be the set of $d \times k$ matrices with each column having $2$-norm at most $r$. Similarly, let $\overline{\mathcal{M}}_{d,k}(r)$ be the set of $d \times k$ matrices with each column having $2$-norm \emph{equal} to $r$. Let $\mathcal{O}_{d,k}$ be the set of $d \times k$ orthonormal matrices. Let $\mathcal{P}_k$ denote the set of $k \times k$ rotation matrices, i.e. $\mathcal{P}_k = \{ P \in \mathbb{R}^{k \times k} : P^T = P^{-1},~ \det(P) = 1\}$.

For a $d \times k$ matrix $M$, let $M^+_{ij} \triangleq \max \{M_{ij}, 0\},$
for $i \in [d]$ and $j \in [k]$, i.e. $M^+$ is the matrix formed from $M$ by replacing each negative entry by $0$. 
%Let $\mathcal{M}_{d,k}^+(s) \triangleq  \{(QR)^+ :  Q \in \mathcal{O}_{d,k}, R \in \overline{\mathcal{M}}_{k,k}(s)\}.$ 
Similarly, for a vector $x \in \mathbb{R}^k$, let $x^+$ denote the positive part of $x$, i.e. $x^+_i = \max\{x_i, 0\}$. For a matrix $M \in \mathbb{R}^{d \times k}$ and $I \subseteq [d]$, let 
\begin{align*}
M(I)_{ij} = \begin{cases} 
M_{ij} & i \in I\\
0 & i \not \in I.
\end{cases}
\end{align*}
In other words, the matrix $M(I)$ is formed from $M$ by zeroing all rows with index not belonging to $I$.
Similarly, for a vector $x \in \mathbb{R}^d$, let $x(I)_i = \mathbbm{1}_{i \in I} x_i$. Note that $(M(I))^T x = M^T (x(I))$. 

Let $\Vert \cdot \Vert_p$ also denote the component-wise $p$-norm of a matrix, i.e. for a $d \times k$ matrix $M$, we have $\Vert M \Vert_p^p = \sum_{i=1}^d \sum_{j=1}^k M_{ij}^p.$ 
The Frobenius norm $\Vert M \Vert_F$ is equal to $\Vert M \Vert_2$ under this notation. 
%Finally, $\langle A, B \rangle = Tr(A^TB)$ denotes the trace inner product of two matrices. \todo{perhaps remove this last thing}

We say a function $f: \mathbb{R}^k \to \mathbb{R}$ is \emph{$l$-Lipschitz} if for every $x, y \in \mathbb{R}^k$ it holds that 
\[|f(x) - f(y)| \leq l \Vert x - y\Vert_2.\]
We say that $f:\mathbb{R}^k \to \mathbb{R}$ is \emph{coordinate-wise monotone} if for all $x, y \in \mathbb{R}^k$ with $x_i \leq y_i$ for each coordinate $i$, it holds that $f(x) \leq f(y)$. In other words, $f$ is coordinate-wise monotone if it is monotone with respect to the Euclidean partial order. Fix $b > 0$. Let $\mathcal{C}(b)$ be the set of coordinate-wise monotone functions $f : \mathbb{R}^k \to [0,b]$, and let $\mathcal{L}_1(b)$ be the set of $1$-Lipschitz coordinate-wise monotone functions $f : \mathbb{R}^k \to [0,b]$. Note that $\mathcal{L}_1(b) \subset \mathcal{C}(b)$.

For a matrix $\beta$ and function $f$, write $(f \circ \beta)(x) \triangleq f(\beta^T x)$. Finally, let $L(x,y,f) \triangleq (f(x) - y)^2$ be the loss function we consider.

\subsection{The Model}\label{section:model}
We now describe the model. Some of the assumptions are carried from \cite{Yang2017}. All parameters except the dimension $d$ are considered constant. 

Let $\beta^{\star} \in \overline{\mathcal{M}}_{d,k}(r)$ be a $d \times k$ matrix of rank $k$, where each column has $2$-norm equal to $r$. Assume that $\beta^{\star}$ is $s^{\star}$-row sparse, meaning that $\beta^{\star}$ has at most $s^{\star}$ nonzero rows. Let $I^{\star} \subset[d]$ be the set of non-zero rows of $\beta^{\star}$, so that $\beta^{\star}(I^{\star}) = \beta^{\star}$. Since $\beta^{\star}$ has full column rank, we can write $\beta^{\star} = Q^{\star} R^{\star}$ as its $QR$ decomposition, where $Q^{\star} \in \mathcal{O}_{d,k}$ and $R^{\star} \in \overline{\mathcal{M}}_{k,k}(r)$ is invertible. We further assume that $\beta^{\star} \geq 0$ entrywise. 
%This constraint is natural in many applications, and rules out degenerate cases. Note that $\beta^{\star +} = \beta^{\star}$ by this assumption.

Let $p_0$ be a twice-differentiable density supported on $\mathcal{X} \subset{R}$.
%with support having Lebesgue measure $\lambda$, i.e. $\int_{x \in \mathcal{X}} dx = \lambda$. 
Let $p^{\star} = \max_{x \in \mathbb{R}} p_0(x)$. Further assume that $\mathcal{X} \subseteq [-C, C]$. Let $X \in \mathbb{R}^d$ be a random variable with density $f_X(x) = \prod_{i=1}^d p_0(x_i)$. 
%By assumption, $|X_i| \leq C$ almost surely, for all $i \in [d]$. 
We additionally assume that $\mathbb{E}[X] = 0$. This is without loss of generality, as we could treat the random variable $X - \mathbb{E}[X] $ with support contained in the set $[-2C, 2C]^d$.

Let $s_0(x) = \frac{p_0'(x)}{p_0(x)}$ for $x \in \mathcal{X}$. Let $f^{\star}\in \mathcal{L}_1(b)$ be a twice-differentiable function. We assume that $\mathbb{E}\left[ \nabla^2 f^{\star}(\beta^{\star T} X) \right] \succ 0$, a restriction that ensures that estimation of $\beta^{\star}$ is information-theoretically feasible \cite{Yang2017}. Let $\rho_0$ be the smallest eigenvalue of $\mathbb{E} \left[ \nabla^2 f^{\star}(\beta^{\star T} X)\right]$. Note that since $\beta^{\star}$ has a constant number of columns and a constant number of nonzero rows, the value $\rho_0$ is itself a constant.

The model is
\begin{align}
Y = (f^{\star} \circ \beta^{\star})(X) + Z. \label{eq:model}
\end{align}
%The columns of $\beta^{\star}$ are referred to as the \emph{index vectors} and $f^{\star}$ is referred to as the \emph{transfer function}. 
Here $Z$ is independent from $X$ and satisfies $\mathbb{E}[Z] = 0$. We assume that $|Z| \leq \eta$ almost surely so that $Y \in \mathcal{Y} \triangleq [-\eta, b + \eta]$ almost surely. Let $F(x,y)$ denote the joint density of $X$ and $Y$. We make a mild distribution assumption, which is that there exists $\theta$ such that $\mathbb{E}[s_0(X)^6] \leq \theta$ and $\mathbb{E}[Y^6] \leq \theta$. Note that since $Y \in [-\eta, b + \eta]$, then $Y^6 \leq (b + \eta)^6$ almost surely.

Given i.i.d. samples $(X_1, Y_1), \dots, (X_n, Y_n)$ drawn from the model \eqref{eq:model}, our goal is to estimate the function $f^{\star} \circ \beta^{\star}$, which is an element of the function class
\[\overline{\mathcal{F}}_{d,k} \triangleq \left\{f \circ \beta(I) : f \in \mathcal{L}_1(b), I \subset [d], |I| = s^{\star}, \beta(I) \in \overline{\mathcal{M}}_{d,k}(r)\right\}.\] 
%Let $\mathcal{F}_{d,k} \triangleq \left\{f \circ \beta : f \in \mathcal{L}_1(b), \beta \in \mathcal{M}_{d,k}(r)\right\}.$
%In fact, the following holds.
\begin{proposition}\label{proposition:scaling}
Let $\mathcal{F}_{d,k} \triangleq \left\{f \circ \beta(I) : f \in \mathcal{L}_1(b), I \subset [d], |I| = s^{\star}, \beta(I) \in \mathcal{M}_{d,k}(r)\right\}.$ It holds that
$\mathcal{F}_{d,k} = \overline{\mathcal{F}}_{d,k}$.
\end{proposition}
By Proposition \ref{proposition:scaling}, the model captures $\beta^{\star} \in \mathcal{M}_{k,k}(r) \subset \overline{\mathcal{M}}_{k,k}(r) $ as well. Observe that for $l > 0$,
\[f \circ \beta \equiv  f(lx) \circ \frac{\beta}{l}.\]
By this identity, the assumption that $f^{\star}$ is $1$-Lipschitz and $\beta^{\star}$ has columns of norm $r$ is without loss of generality; the assumption is equivalent to the assumption that $f^{\star}$ is $l$-Lipschitz and $\beta^{\star}$ has columns of norm $\nicefrac{r}{l}$.

\subsection{Outline}
We combine the results of two recent papers in order to design an algorithm for estimation in MMI models. \cite{Yang2017} provide an algorithm for estimation of $Q^{\star}$ up to rotation given samples from the model \eqref{eq:model}. In other words, they find $Q$ such that $QP \approx Q^{\star}$ for some rotation matrix $P$. In Section \ref{section:estimation-Q-star}, we summarize the approach of \cite{Yang2017} to estimate the matrix $Q^{\star}$, up to rotation.

Informally, observe that if $QP \approx Q^{\star}$ and $R \approx PR^{\star}$, then $QR \approx Q^{\star} R^{\star}$. Given a $Q$ that approximates $Q^{\star}$ up to rotation, it remains to find $R \in \overline{\mathcal{M}}_{k \times k}(r)$, an index set $I$ of cardinality $s^{\star}$, as well as a function $f$.  Thus, the estimation of $Q^{\star}$ up to rotation reduces the high-dimensional estimation problem to a lower-dimensional problem. 

Our approach is to form a collection of candidate $k \times k$ matrices (Section \ref{section:near-net}). For each candidate matrix, we find the optimal index set and accompanying coordinate-wise monotone function (Section \ref{section:sparse}).  We call the problem of finding the optimal index set $I$ and coordinate-wise monotone function $f$ the Sparse Matrix Isotonic Regression Problem. We extend the recent work of \cite{Gamarnik2019}, who consider a related isotonic regression problem.
%, and provide guarantees for estimation in terms of $L_2$ loss. 

In Section \ref{section:main}, we tie together the results of the previous three sections in order to provide an algorithm for estimation in the high-dimensional monotone multi-index model. The algorithm finds a function of the form $f \circ (QR)^+(I)$ minimizing the sample loss over the candidate matrices $R$. Here $I$ is an index set and $f$ is a coordinate-wise monotone function obtained by solving the Sparse Matrix Isotonic Regression Problem. We give estimation guarantees for our algorithm in terms of $L_2$ loss. Finally, Section \ref{section:conclusion} outlines some future directions.

%we introduce the sparse matrix isotonic regression problem, and provide an algorithm for solving it.  
%Next, in Section \ref{section:low-dimensional}, we assume we have a matrix $Q$ that differs from $Q^{\star}$ by only a rotation. Namely, $QP = Q^{\star}$ for some rotation matrix $P$. It remains to estimate the product $PR^{\star}$ as well as the function $f^{\star}$, which is now a low-dimensional problem. 

\subsection{Contributions}
Let $\Vert f^{\star} \circ Q^{\star} R^{\star} -g \Vert_2$ denote the expected $L_2$ loss of a function $g$ with respect to the ground truth:
\[\Vert f^{\star} \circ Q^{\star} R^{\star} - g \Vert_2^2 \triangleq \int_{x \in \mathcal{X}} \left[ (f^{\star} \circ Q^{\star} R^{\star})(x) - g(x) \right]^2 f_X(x) dx .\]
Let $z(\epsilon_1, \epsilon_2, C) \triangleq 2 \eta C \sqrt{k} \left(\epsilon_1 + \epsilon_2 r \right) + C^2 k \left(\epsilon_1 + \epsilon_2 r \right)^2.$ The main result of our paper is the following.% (where Algorithm \ref{alg:full} appears in Section \ref{section:main}). 
\begin{theorem}\label{theorem:sample-main-result}
Fix $\epsilon> 0$. Let $\delta = \delta(\epsilon)$ be the solution to $z(\delta, \delta, C) = \nicefrac{\epsilon}{2}$. 
%Fix \[N_0 = \left \lceil \frac{\log\left(\frac{\epsilon}{3k}\right)}{\log \left(1-   \left| S_r^{k-1} \cap B\left(e_1, \frac{\delta}{\sqrt{k}}\right)\right| \left|S_r^{k-1} \right|^{-1} \right)} \right \rceil.\]
Suppose $d \geq \sqrt{\frac{3}{\epsilon}}$. Given $n$ independent samples $(X_i, Y_i)_{i=1}^n$ from the model \eqref{eq:model}, there exists an algorithm that produces an estimator $ f_n \circ M_n^+(I_n)$ such that
%Assuming that $n = \omega(\log(d))$, then there exists $n_0 = n_0(s^{\star}, k, r, b, C, p^{\star})$ such that if $n \geq n_0$, then
\begin{align*}
\mathbb{P}\left( \Vert f_n \circ M_n^+(I_n) - f^{\star} \circ Q^{\star}R^{\star} \Vert_2^2 \geq \epsilon \right) \leq \epsilon.
\end{align*}
whenever $n \geq C_1 \log(d) + C_2$, for constants $C_1$ and $C_2$ depending on $C$, $b$, $s^{\star}$, $p^{\star}$, $k$, $\rho_0$, $\theta$, and $\eta$.
\end{theorem}
The significance of this result is that we can estimate the ground truth function with a high degree of accuracy even when the dimension $d$ is much larger than the number of samples $n$. The proofs are deferred to the supplementary material, with the exception of the proof of our key result, Theorem \ref{theorem:main-result}, that immediately implies Theorem \ref{theorem:sample-main-result}.

\section{Estimation of $Q^{\star}$}\label{section:estimation-Q-star}
We summarize the work of \cite{Yang2017}, who estimate $Q^{\star}$ up to rotation. The approach of \cite{Yang2017} uses the second-order Stein condition. For $x \in \mathbb{R}^d$, let $T(x)$ be the $d \times d$ matrix defined as follows.
\begin{align*}
T(x)_{ij} &= \begin{cases}
s_0(x_i) s_0(x_j) & i \neq j\\
s_0(x_i)^2 - s_0'(i) & i = j.
\end{cases}
\end{align*}
\cite{Yang2017} show the identity $\mathbb{E}\left[ Y \cdot T(X) \right] = Q^{\star} D_0 Q^{\star}$, where $D_0 = \mathbb{E}\left[ \nabla^2 f(\beta^{\star T} X) \right]$. Therefore, one can estimate $Q^{\star}$ from the leading eigenvectors of the sample average of the quantity $Y \cdot T(X)$. \cite{Yang2017} use a robust estimator for $Y \cdot T(X)$. For $\tau > 0$, define the truncated random variables
\[\tilde{Y}_i \triangleq \text{sign}(Y_i) \cdot \min \{|Y_i|, \tau\} \text{ and } \tilde{T}_{jk}(X_i) \triangleq \text{sign}\left( T_{jk}(X_i)\right) \cdot \min \{ \left|T_{jk}(X_i)\right|, \tau^2\}.\] 
The robust estimator is given by
\[\tilde{\Sigma} = \tilde{\Sigma}(\tau) \triangleq \frac{1}{n} \sum_{i=1}^n \tilde{Y}_i \cdot \tilde{T}(X_i).\]
\cite{Yang2017} propose the following approach to estimate $Q^{\star}$ up to rotation.
%The following semidefinite program with parameters $\tau > 0$ and $\lambda >0$ is solved in \cite{Yang2017}.
\begin{breakablealgorithm}
\caption{Estimation of $Q^{\star}$ \cite{Yang2017}}\label{alg:SDP}
\begin{algorithmic}[1]
\Require{Values $(X_1, Y_1), \dots, (X_n, Y_n)$, $\tau > 0$, and $\lambda >0$}
\Ensure{$\hat{Q} \in \mathcal{O}_{d,k}$}
\State Compute the estimator $\tilde{\Sigma}(\tau)$ using the samples $(X_i, Y_i)_{i=1}^n$.
\State Solve the following optimization problem.
\begin{align}
%\max ~~ &\langle W, \tilde{\Sigma}(\tau) \rangle + \lambda \Vert W \Vert_1 \label{eq:SDP-1}\\
\max ~~ &Tr(W^T~\tilde{\Sigma}(\tau))  + \lambda \Vert W \Vert_1 \label{eq:SDP-1}\\
\text{s.t.} ~~ &0 \preceq W \preceq I_d \label{eq:SDP-2}\\
&Tr(W) = k \label{eq:SDP-3}.
\end{align}
\State Let $\hat{Q}$ be the matrix whose columns are the $k$ leading eigenvectors of $\hat{W}$.
\end{algorithmic}
\end{breakablealgorithm}

\begin{theorem}[Adapted from Theorem 3.3 from \cite{Yang2017}]\label{theorem:yang}
Let $\tau = \left( \frac{3 \theta n}{2 \log d} \right)^{\frac{1}{6}}$ and $\lambda = 10 \sqrt{\theta \frac{\log d}{n}}$. Under the assumptions of Section \ref{section:model}, With probability at least $1 - d^{-2}$, Algorithm \ref{alg:SDP} applied to samples $(X_i, Y_i)_{i=1}^n$, $\tau$, and $\lambda$ produces an estimator $\hat{Q}$ satisfying
\[ \inf_{P \in \mathcal{P}_k} \Vert \hat{Q}P - Q^{\star} \Vert_F \leq \frac{1}{\rho_0} 4 \sqrt{2} s^{\star} \lambda.\]
%with probability at least $1 - d^{-2}$.
\end{theorem}

Assuming that $d$ grows with $n$, Theorem \ref{theorem:yang} shows that with high probability as $n \to \infty$, the estimate of $Q^{\star}$ is correct up to rotation, with error on the order of $\sqrt{\nicefrac{\log (d)}{n}}$. 

\section{Construction of a Near-Net}\label{section:near-net}
Fix $\delta > 0$. We construct a random set of matrices $\mathcal{R}$ that will serve to approximate the set of $k \times k$ matrices with column norm $r$, with respect to the Frobenius norm. Given $\epsilon, \delta > 0$, we choose the cardinality of the set of approximating matrices so that a fixed matrix from the set $\overline{\mathcal{M}}_{k,k}(r)$ is $\epsilon$-close to some element of $\mathcal{R}$ with probability $1- \delta$.
%The set $\mathcal{R}$ will have the property that a fixed matrix from the set $\overline{\mathcal{M}}_{k,k}(r)$ is $\delta$-close to some element of $\mathcal{R}$ with a high probability. 
For this reason, we call $\mathcal{R}$ a \emph{near-net}. To construct $\mathcal{R}$, we first construct a random set of vectors $\mathcal{R}_0$ by choosing $N_0$ vectors from the uniform measure of all vectors of $2$-norm $r$. In other words, each element from $\mathcal{R}_0$ is chosen from the uniform measure on the $k$-dimensional sphere of radius $r$, denoted by $S_{r}^{k}$. We may sample uniformly using $k$ independent random variables $Z_1, \dots, Z_{k} \sim \mathcal{N}(0,1)$. The random vector 
\[\frac{1}{\sum_{i=1}^{k} Z_i^2} \left(Z_1, \dots, Z_{k}\right)\]
is uniformly distributed on the surface of $S_r^{k}$. Finally, we then construct the set $\mathcal{R}$ as the set of all matrices with columns belonging to $\mathcal{R}_0$. Then $|\mathcal{R}| = N_0^k$.

For a vector $x \in \mathbb{R}^k$ and $\epsilon > 0$, let $B(x, \epsilon) \triangleq \{y : \Vert x - y \Vert_{2} \leq \epsilon\} $ denote the ball of radius $\epsilon$ around $x$ with respect to the $2$-norm.
\begin{lemma}\label{lemma:near-net}
Let $\epsilon > 0$. Consider the near-net $\mathcal{R}(N_0)$ described above, and let $M \in \overline{\mathcal{M}}_{k,k}(r)$ be a fixed matrix. With probability at least
\begin{align}
1 - k\left(1-   \left| S_r^{k} \cap B\left(e_1, \frac{\epsilon}{\sqrt{k}}\right)\right| \left|S_r^{k} \right|^{-1} \right)^{N_0}, \label{eq:probability}
\end{align}
%\[1 - k \left(1-  (2\epsilon)^{k-1} \left| S_r^{k-1} \right|^{-1}\right)^{N_0},\]
there exists $\overline{R} \in \mathcal{R}(N_0)$ such that $\Vert M - \overline{R}\Vert_{F} \leq \epsilon$, where $e_1 = (1, 0, \dots, 0) \in \mathbb{R}^k$ and $| A |$ denotes the measure of a set $A$.
\end{lemma}
\begin{remark}
As $N_0 \to \infty$, the probability \eqref{eq:probability} goes to $1$.
\end{remark}
Our random construction is simple to implement. While deterministic constructions are possible, they are much more complex (see \cite{Dumer2007}).
%We note that another natural choice would be to construct a true $\delta$-net of $k \times k$ matrices with column norm $r$, with respect to the $2$-norm. As in the random near-net, we may reduce to the problem of constructing a net of vectors of column norm $r$, with respect to the $2$-norm. The problem of finding an optimal net of vectors with respect to the $2$-norm is an open problem \cite{Dumer2007}. Finding a small near-net is significantly easier than finding a small true net. For this reason, we opt to use a random near-net.

\section{Sparse Matrix Isotonic Regression}\label{section:sparse}
Recently, Gamarnik and Gaudio introduced the Sparse Isotonic Regression model  \cite{Gamarnik2019}.
%, defined as follows \cite{Gamarnik2019}. Let $A \subset [d]$. For $x, y \in \mathbb{R}^d$, we write $x \preceq_A y$ if $x_i \leq y_i$ for all $i \in A$. We say that a function $f$ is \emph{$s$-sparse coordinate-wise monotone} if there exists a set $A \subset [d]$ such that $|A| = s$ and $x \preceq_A y \implies f(x) \leq f(y)$ for all $x, y \in \mathbb{R}^d$. In other words, $f$ respects the partial order induced by an $s$-sparse subset of the coordinates. We call the set $A$ the \emph{active coordinates}. \cite{Gamarnik2019} provides estimation algorithms for sparse isotonic regression. 
We now introduce a new related model, \emph{Sparse Matrix Isotonic Regression}. We are given a $d \times k$ matrix $M$ with nonnegative entries as well as samples $(X_i, Y_i)_{i=1}^n$. For a given sparsity level $s \in \mathbb{N}$ and bound $b > 0$, our goal is to find a set $I \subset [d]$ with cardinality $s$, and a coordinate-wise monotone function $f : \mathbb{R}^k \to [0,b]$ minimizing $\sum_{i=1}^n \left( Y_i - (f \circ M(I))(X_i) \right)^2.$
Our approach is to estimate the function values at the points $X_1, \dots, X_n$ and interpolate. We emphasize that we do not require the function $f$ to be $1$-Lipschitz.

The Integer Programming Sparse Matrix Isotonic Regression algorithm finds the optimal index set and function values on a given set of points, given a matrix $M$ with nonnegative entries. Binary variables $v_l$ determine the index set $I$. The variables $F_i$ represent the estimated function values at data points $X_i$. Auxiliary variables $z_{ij}$ and $q_{ijp}$ are used to model the monotonicity constraints. The function that is returned is an interpolation of the points $(M(I)^T X_i, F_i)_{i=1}^n$.
 \begin{breakablealgorithm}
\caption{Integer Programming Matrix Isotonic Regression}\label{alg:integer-program}
\begin{algorithmic}[1]
\Require{Values $(X_1, Y_1), \dots, (X_n, Y_n)$, sparsity level $s$, $M \geq 0 \in \mathbb{R}^{d \times k}$, $C > 0$, $b > 0$}
\Ensure{An index set $I \subset [d]$ satisfying $|I| = s$; a coordinate-wise monotone function $f: \mathbb{R}^k \to [0,b]$}
\State Let $B = 2C \sum_{l=1}^d \sum_{p=1}^k M_{lp}$. Let $$\mu = \min\{M_{lp} > 0 : l \in [d], p \in [k]\} \cdot \min_{i, j \in [n], i \neq j} |X_{il} - X_{jl}|.$$
\State Solve the following optimization problem.
\small
\begin{align}
&\min_{v, F, z} \sum_{i=1}^n \left(Y_i - F_i\right)^2 \label{eq:objective}\\
\text{s.t. } &\sum_{l=1}^d v_l = s  \label{eq:sparsity}\\
& F_i - F_j \leq b z_{ij} \label{eq:ordering}\\
&\sum_{p=1}^k q_{ijp} \geq z_{ij} & \forall i,j \in [n] \label{eq:auxiliary}\\
&\sum_{l=1}^d v_l M_{lp} (X_{il} - X_{jl}) - \frac{\mu}{2} \geq -\left(B + \frac{\mu}{2}\right)(1- q_{ijp} ) &\forall i,j \in [n], p \in [k] \label{eq:monotonicity}\\
&v_l \in \{0,1\} &\forall l \in [d] \nonumber \\
&F_i \in [0,b] &\forall i \in [n]  \nonumber \\
& z_{ij} \in \{0,1\} & \forall i, j \in [n] \nonumber\\
& q_{ijp} \in \{0,1\} & \forall i, j \in [n], p \in [k] \nonumber
\end{align}
\normalsize
\State Let $I_n = \{ l \in [d] : v_l = 1\}$. Let $\hat{f}_n(x) = \max \{F_i : M(I_n)^TX_i \preceq x\}$ and $\hat{f}_n(x) = 0$ if $\{M(I_n)^TX_i \preceq x\}_{i=1}^n = \emptyset$
\State Return $(I_n, \hat{f}_n)$. 
\end{algorithmic}
\end{breakablealgorithm}
\begin{proposition}\label{proposition:optimal}
Suppose $X_i \in [-C, C]^d$ for $i \in [n]$. On input $(X_i, Y_i)_{i=1}^n, s, M, C, b$, Algorithm \ref{alg:integer-program} finds a function $\hat{f}_n \in \mathcal{C}(b)$ and index set $I_n$ that minimize the empirical loss $\sum_{i=1}^n L(X_i, Y_i, f \circ M(I))$, over functions $f \in \mathcal{C}(b)$ and index sets $I$ with cardinality $s$. 
\end{proposition}
%The IPIR algorithm is a convex quadratic mixed-integer program with $d$ binary variables. \cite{Gamarnik2019} provides another approach that uses linear programming to estimate the active coordinates. However, the algorithm is not guaranteed to find an optimal solution, and is therefore challenging to analyze in this setting.
The integer program in Algorithm \ref{alg:integer-program} has a convex objective and linear constraints. While integer programming is NP-hard in general, modern solvers achieve excellent performance on such problems. We note that it is possible to ensure that the function $\hat{f}_n$ be $1$-Lipschitz in addition to coordinate-wise monotone, by modifying the optimization problem in Algorithm \ref{alg:integer-program}. However, the resulting optimization problem is an integer program with nonlinear constraints, a less tractable formulation. For further details, please see Section \ref{section:Lipschitz} in the supplementary material. 

\section{Estimation guarantees for the MMI Model}\label{section:main}
In this section, we provide estimation guarantees for the model $Y = (f^{\star} \circ Q^{\star} R^{\star})(X) + Z.$ Let $N \triangleq 2n$ be the sample size. We use $n$ samples for estimation of $Q^{\star}$ (up to rotation), obtaining a matrix $Q_n$, and another $n$ samples to obtain a matrix $R_n$ a function $f_n$, and index set $I_n$. The final result is an estimated function $f_n\circ (Q_nR_n)^+(I_n)$. 

We now outline the approach. First, by Theorem \ref{theorem:yang}, the matrix $Q_n$ obtained from the semidefinite programming approach satisfies $\Vert Q_nP_n - Q^{\star} \Vert_{F} \leq \frac{1}{\rho_0} 4 \sqrt{2} s^{\star} \lambda$ with probability at least $1 - d^{-2}$. We use this matrix $Q_n$ to estimate $R_n$, $f_n$ and $I_n$, assuming that $\Vert Q_nP_n - Q^{\star} \Vert_{F} \leq \frac{1}{\rho_0} 4 \sqrt{2} s^{\star} \lambda$ for some unknown rotation matrix $P_n$. 
%Note that since we are using independent samples to estimate $R_n$ and $f_n$, conditioning on this property of $Q_n$ does not affect the subsequent probabilistic analysis. 
The joint estimation of $(R_n, f_n, I_n)$ is intractable; instead, we create a net of candidate matrices from the set $\overline{\mathcal{M}}_{k,k}(r)$. For each net element $R$, we apply Algorithm \ref{alg:integer-program} to find the optimal pair $(f_R, I_R)$ minimizing the empirical loss $\sum_{i=n+1}^{2n} L(X_i, Y_i, f \circ (Q_n R)^+(I))$. Finally, we output the best combination over the net elements. 

Recall that $f^{\star} \circ \beta^{\star} \in \overline{\mathcal{F}}_{d,k}(r)$. While $Q_n \in \mathcal{O}_{d,k}$ and $R_n \in \overline{\mathcal{M}}_{k,k}(r)$,
the matrix $(Q_nR_n)^+(I_n)$ may not be an element of $\overline{\mathcal{M}}_{d,k}(r)$. Further, the estimated function $f_n$ may not be $1$-Lipschitz. Nevertheless, we are able to give an $L_2$ loss guarantee, as we will see in the proof of Theorem \ref{theorem:main-result}.

%Such a net can be constructed as follows. We create a $\delta$-net $\mathcal{R}_0$ of vectors in $\mathbb{R}^k$ with $2$-norm equal to $r$, with respect to the $2$-norm. The problem of finding such a net $\mathcal{R}_0$ can be reduced to the related problem of covering a sphere with spherical caps. Finding the smallest covering is an open problem. 
%Algorithms for constructing such nets, as well as upper bounds on their size, can be found in \cite{Dumer2007}. 
%The results are in the asymptotic regime in the dimension of the vector. 

\begin{breakablealgorithm}
\caption{MMI Regression}\label{alg:full}
\begin{algorithmic}[1]
\Require{$N_0 \in \mathbb{N}$, values $(X_1, Y_1), \dots, (X_N, Y_N)$, $C > 0$, $b > 0$, $\tau > 0$, and $\lambda > 0$}
\Ensure{$f_n \in \mathcal{L}_1(b)$, $Q_n \in \mathcal{O}_{d,k}$, $\overline{R}_n \in \overline{\mathcal{M}}_{k,k}(r)$, and $I_n \in [d] : |I_n| = s^{\star}$}
\State Construct a random near-net $\mathcal{R}(N_0)$.
%\State Let $\tau = \left( \frac{3 Mn}{2 \log d} \right)^{\frac{1}{6}}$ and $\lambda = 10 \sqrt{M \frac{\log d}{n}}$,
\State Produce an estimate $Q_n$ using Algorithm \ref{alg:SDP} applied to $(X_i, Y_i)_{i=1}^n$, $\tau$, and $\lambda$.
\ForEach {$R \in \mathcal{R}$} 
\State Apply Algorithm \ref{alg:integer-program} to input $(X_i, Y_i)_{i=n+1}^{2n}$, $s^{\star}$, $(Q_nR)^+$, $C$, and $b$, obtaining the function $f_R$ and index set $I_R$. 
\EndFor
\State Return the tuple $(f_R, Q_n, R, I_R)$ with the smallest empirical loss.
\end{algorithmic}
\end{breakablealgorithm}

The following result provides an upper bound on the error associated with the estimator from Algorithm \ref{alg:full}. Our main result, Theorem \ref{theorem:sample-main-result}, easily follows from Theorem \ref{theorem:main-result}.
\begin{theorem}\label{theorem:main-result}
Let $X \in \mathbb{R}^d$ be a random variable with independent entries of density $p_0 \leq p^{\star}$ and support contained within the set $[-C, C]^d$. Assume that $f^{\star} : \mathbb{R}^k \to [0,b]$ for $b > 0$. Fix $n$. Let $\tau = \left( \frac{3 \theta n}{2 \log d} \right)^{\frac{1}{6}}$ and $\lambda = 10 \sqrt{\theta \frac{\log d}{n}}$. Let $\epsilon > 0$ and let $\delta > 0$ be such that $\epsilon > z\left(\delta, \frac{1}{\rho_0} 4 \sqrt{2} s^{\star} \lambda, C \right).$
%\[\epsilon > (E_n \sqrt{d} + \delta) C k^{\frac{3}{2}}  \left( 2\eta +(E_n \sqrt{d}+ \delta) C k^{\frac{3}{2}} \right)\] 
Let $(f_n, Q_n, \overline{R}_n, I_n)$ be the result of applying Algorithm \ref{alg:full} on inputs $N_0$, $(X_1, Y_1), \dots, (X_{2n}, Y_{2n})$, $C$, $b$, $\tau$, and $\lambda$. Let $M_n = Q_n \overline{R}_n$. Then
\small
\begin{align*}
\mathbb{P}\left(\Vert f_n \circ M_n^+(I_n) - f^{\star} \circ Q^{\star} R^{\star} \Vert_2^2 \geq \epsilon \right) &\leq k\left(1-   \left| S_r^{k-1} \cap B\left(e_1, \frac{\delta}{\sqrt{k}}\right)\right| \left|S_r^{k-1} \right|^{-1} \right)^{N_0}+ \frac{1}{d^2} \\
&~~+ 4\binom{d}{s^{\star}} N_0^k \exp \left[ \left( \frac{2 \log(2)b}{\alpha} +2^{\frac{b}{\alpha}} (4p^{\star}C)^{s^{\star}} \right) n^{\frac{s^{\star}-1}{s^{\star}}}  -\frac{\epsilon_0^2 n}{2^9 b^2} \right], %\label{eq:bound-main}
\end{align*}
\normalsize
where $\epsilon_0 = \epsilon - z\left(\delta, \frac{1}{\rho_0} 4 \sqrt{2} s^{\star} \lambda , C \right)$ and $\alpha =  \frac{1}{64}\epsilon_0 (b+\eta)^{-1}$.
\end{theorem}

The following results are used in the proof of Theorem \ref{theorem:main-result}. Lemma \ref{lemma:net-sensitivity-main} establishes a sensitivity result. The Lipschitz assumption on $f^{\star}$ is a key element in proving Lemma \ref{lemma:net-sensitivity-main}.
\begin{lemma}\label{lemma:net-sensitivity-main}
Let $X \in \mathbb{R}^d$ be a random variable with independent entries of density $p_0 \leq p^{\star}$ and support contained within the set $[-C, C]^d$.
Suppose that $R \in \overline{\mathcal{M}}_{k,k}(r)$ satisfies $\Vert PR^{\star} -R\Vert_{F} \leq \epsilon_1$. Suppose also that $T \in \mathcal{O}_{d,k}$ satisfies $\Vert TP - Q^{\star} \Vert_F \leq \epsilon_2$ for some rotation matrix $P \in \mathcal{P}_{k,k}$. Then
\[\int L(x,y, f^{\star} \circ (TR)^+(I^{\star})) dF(x,y) - \int L(x,y, f^{\star} \circ Q^{\star} R^{\star}) dF(x,y) \leq z(\epsilon_1, \epsilon_2, C) .\]
\end{lemma}

Lemma \ref{lemma:norm-integral} relates the $2$-norm difference of two functions to a difference of integrals.
\begin{lemma}\label{lemma:norm-integral}
Let $g$ be any function from $\mathbb{R}^k$ to $\mathbb{R}$. Then
\[\left \Vert g  - f^{\star} \circ Q^{\star} R^{\star} \right \Vert_2^2 = \int L(x,y, g) dF(x,y) - \int L(x,y, f^{\star} \circ Q^{\star} R^{\star}) dF(x,y).\]
\end{lemma}

Fix $b > 0$. For $T \in \mathcal{O}_{d,k}$, $R \in \overline{\mathcal{M}}_{k,k}(r)$, and $I \subset [d]$ with $|I| = s^{\star}$, let 
\[\mathcal{G}(T, R, I) = \{f \circ (TR)^+(I) : f \in \mathcal{C}(b)\} \text{ and } \mathcal{G}(T, \mathcal{R}) \triangleq \cup_{R \in \mathcal{R}}\cup_{I \subset [d] : |I| = s^{\star}} \mathcal{G}(T, R, I).\]
 %For a set $\mathcal{R}$ of matrices in $\overline{\mathcal{M}}_{k,k}(r)$ and $\delta  >0$, let
%\[\mathcal{R}^+(T, \delta) \triangleq \{ R \in \mathcal{R} : TR \in \mathcal{M}^+(r, k\delta)\}.\]
%\[\mathcal{G}^+(T, \mathcal{R}, \delta) \triangleq \cup_{R \in \mathcal{R}^+(T, \delta)} \mathcal{G}(T, R).\] 
We see that Algorithm \ref{alg:full} optimizes the empirical loss over functions in $\mathcal{G}(Q_n, \mathcal{R})$. We follow a VC entropy approach to give an $L_2$ loss bound for the function $f_n \circ M_n^+(I_n)$ estimated by Algorithm \ref{alg:full}.

\begin{definition}\label{definition:net}
Let $\mathcal{F}$ be a class of functions from $\mathbb{R}^d$ to $\mathbb{R}$. Given
\[(x_1, y_1), \dots, (x_n, y_n) \in \mathbb{R}^d \times \mathbb{R},\]
let
\[\mathbf{L}_{\mathcal{F}}\left((x_1, y_1), \dots, (x_n, y_n)\right) \triangleq \left\{ \left( L(x_1, y_1, f), \dots, L(x_n, y_n, f) \right) : f \in \mathcal{F} \right\}.\]
In other words, $\mathbf{L}_{\mathcal{F}}$ is the set of loss vectors formed by ranging over functions $f$ in the class $\mathcal{F}$. Let $N_{\mathcal{F}}\left((x_1, y_1), \dots, (x_n, y_n), \epsilon \right)$ denote the size of the smallest $\epsilon$-net for $\mathbf{L}_{\mathcal{F}}\left((x_1, y_1), \dots, (x_n, y_n)\right)$, with respect to the $\infty$-norm. In other words, for every $u \in \mathbf{L}_{\mathcal{F}}\left((x_1, y_1), \dots, (x_n, y_n)\right)$, there exists $v \in N_{\mathcal{F}}\left((x_1, y_1), \dots, (x_n, y_n)\right) $ such that $\Vert u - v \Vert_{\infty} \leq \epsilon$. Finally, let
\[N_{\mathcal{F}}(\epsilon, n) \triangleq \mathbb{E}_{X, Y} \left[N_{\mathcal{F}}\left((X_1, Y_1), \dots, (X_n, Y_n), \epsilon \right) \right]\]
be the expected size of the net, where the expectation is over independent samples drawn from the distribution $F(x,y)$ defined above.
\end{definition}

Lemmas \ref{lemma:sample-accuracy-main-helper} and \ref{lemma:net-bound} together provide a probabilistic bound on the difference between expected loss and empirical loss for functions in the class $\mathcal{G}(T,\mathcal{R})$. The nonnegative matrix assumption is crucial for the proof of Lemma \ref{lemma:net-bound}.
\begin{lemma}\label{lemma:sample-accuracy-main-helper}
Let $T \in \mathcal{O}_{d,k}$ and let $\delta > 0$. Let $\mathcal{R} \subset \overline{\mathcal{M}}_{k,k}(r)$. For $\epsilon > 0$,
\begin{align*}
\mathbb{P} \left( \sup_{h \in \mathcal{G}(T, \mathcal{R})}  \left| \int L(x,y, h) dF(x,y) - \frac{1}{n} \sum_{i=1}^n L(X_i, Y_i, h) \right|  \geq  \epsilon \right)
&\leq 4 N_{\mathcal{G}(T, \mathcal{R})}\left(\frac{\epsilon}{16} , n\right) \exp \left(-\frac{\epsilon^2 n}{128 b^2} \right).
\end{align*}
%where $\alpha = \alpha(\epsilon) = \frac{1}{2}\epsilon (b+\eta)^{-1}$.
\end{lemma}

\begin{lemma}\label{lemma:net-bound}
Under the assumptions of Lemma \ref{lemma:sample-accuracy-main-helper}, it holds that
\[N_{\mathcal{G}(T, \mathcal{R})}\left(\epsilon, n\right)  \leq  \binom{d}{s^{\star}} N_0^k \exp \left[ \left( \frac{2 \log(2)b}{\alpha} +2^{\frac{b}{\alpha}} (4p^{\star}C)^{s^{\star}} \right) n^{\frac{s^{\star}-1}{s^{\star}}} \right], \]
%\[N_{\mathcal{G}}(\epsilon, n) \leq |\mathcal{R}| \exp \left[ \left( \frac{2 \log(2) b}{\alpha} + D(t)2^{\frac{b}{\alpha} + 2k} (rC)^k \right) n^{\frac{k-1}{k}} \right],\]
where $\alpha = \frac{1}{2}\epsilon (b+\eta)^{-1}$.
\end{lemma}
With these results stated, we are now ready to prove Theorem \ref{theorem:main-result}.
\begin{proof}[Proof of Theorem \ref{theorem:main-result}]
With probability at least $1 - d^{-2}$, the matrix $Q_n$ satisfies 
\[\Vert Q_n P_n - Q^{\star} \Vert_F \leq \frac{1}{\rho_0} 4 \sqrt{2} s^{\star} \lambda\] 
for some rotation matrix $P_n$ (Theorem \ref{theorem:yang}). For the remainder, we condition on this property of $Q_n$, since we use this matrix on an independent batch of samples $(X_i, Y_i)_{i=n+1}^{2n}$.  
By Lemma \ref{lemma:norm-integral}, 
\begin{align*}
\Vert f_n \circ M_n^+(I_n) - f^{\star} \circ Q^{\star}R^{\star} \Vert_2^2 &= \int L(x,y, f_n \circ M_n^+(I_n)) dF(x,y) - \int L(x,y, f^{\star} \circ Q^{\star} R^{\star}) dF(x,y).
\end{align*}
Recall that $I^{\star}$ is the set of non-zero rows of $\beta^{\star}$. Let $E$ be the event that the near-net $\mathcal{R}$ contains an element $\overline{R} \in \mathcal{R}$ such that $\Vert PR^{\star} - \overline{R}\Vert_{F} \leq \delta$. 
%Observe that 
%\[\Vert \beta^{\star} - Q\overline{R}\Vert_{\infty} =  \Vert QPR^{\star} - Q\overline{R}\Vert_{\infty} \leq k\delta.\]
%Since $\beta^{\star} \geq 0$ entrywise, we have $Q\overline{R} \geq -k\delta$ entrywise. Therefore, $\overline{R}$ is considered by Algorithm \ref{alg:low-dimensional}. 
Conditioned on $E$, let $\overline{R}$ be the (random) matrix that is $\delta$-close to $PR^{\star}$. By Proposition \ref{proposition:optimal}, the function $f_n \circ M_n^+(I_n)$ is optimal over the samples. Therefore, 
$\sum_{i=n+1}^{2n} L(X_i, Y_i, f_n \circ M_n^+(I_n)) \leq \sum_{i=n+1}^{2n} L(X_i, Y_i, f^{\star} \circ (Q_n \overline{R})^+(I^{\star})).$ We have
\begin{align*}
\Vert f_n \circ M_n^+(I_n)- f^{\star} \circ Q^{\star} R^{\star} \Vert_2^2 &\leq \int L(x,y, f_n \circ M_n^+(I_n)) dF(x,y) - \frac{1}{n} \sum_{i=n+1}^{2n} L(X_i, Y_i, f_n \circ M_n^+(I_n)) \\
&+ \frac{1}{n}\sum_{i=n+1}^{2n} L(X_i, Y_i, f^{\star} \circ (Q_n \overline{R})^+(I^{\star})) - \int L(x,y, f^{\star} \circ (Q_n \overline{R})^+(I^{\star}))dF(x,y)\\
&+ \int L(f^{\star} \circ (Q_n \overline{R})^+(I^{\star}))dF(x,y) - \int L(x, y, f^{\star} \circ Q^{\star} R^{\star}) dF(x,y).
\end{align*}
By Lemma \ref{lemma:net-sensitivity-main} applied to $T = Q_n$, $\epsilon_1 = \delta$, and $\epsilon_2 = \frac{1}{\rho_0} 4 \sqrt{2} s^{\star} \lambda$,
\[\int L(f^{\star} \circ (Q_n \overline{R})^+(I^{\star}))dF(x,y) - \int L(x, y, f^{\star} \circ Q^{\star} R^{\star}) dF(x,y) \leq z\left(\delta,  \frac{1}{\rho_0} 4 \sqrt{2} s^{\star} \lambda, C \right).\]
Therefore, 
\begin{align*}
&\mathbb{P}\left( \Vert f_n \circ M_n^+(I_n) - f^{\star} \circ Q^{\star}R^{\star} \Vert_2^2 \geq \epsilon ~\Big|~ E\right) \\
&\leq \mathbb{P} \left(\int L(x,y, f_n \circ M_n^+(I_n))dF(x,y) - \frac{1}{n} \sum_{i=1}^n L(X_i, Y_i, f_n \circ M_n^+(I_n)) \right.\\
&~~~~~~~~~~~~\left.+ \frac{1}{n} \sum_{i=1}^n L(X_i, Y_i, f^{\star} \circ (Q\overline{R})^+(I^{\star})) - \int L(x,y, f^{\star} \circ (Q \overline{R})^+(I^{\star})) dF(x,y) \geq \epsilon_0 ~\Big|~ E\right).
\end{align*}
Observe that the functions $f_n \circ M_n^+(I_n)$ and $f^{\star} \circ (Q_n \overline{R})^+(I^{\star})$ are elements of $\mathcal{G}(Q_n, \mathcal{R})$.
Since the event $E$ is independent from the samples $(X_i, Y_i)$, we apply Lemmas \ref{lemma:sample-accuracy-main-helper} and \ref{lemma:net-bound}. 
\begin{align*}
&\mathbb{P}\left( \Vert f_n \circ M_n^+(I_n) - f^{\star} \circ Q^{\star}R^{\star} \Vert_2^2 \geq \epsilon ~\Big|~E \right) \\
&\leq \mathbb{P} \left(\sup_{h \in \mathcal{G}(Q_n, \mathcal{R})} \left| \int L(x,y, h) dF(x,y) - \frac{1}{n} \sum_{i=1}^n L(X_i, Y_i, h) \right| \geq \frac{\epsilon_0}{2} ~\Big|~E\right)\\
& \leq 4\binom{d}{s^{\star}} |\mathcal{R}| \exp \left[ \left( \frac{2 \log(2)b}{\alpha} +2^{\frac{b}{\alpha}} (4p^{\star}C)^{s^{\star}} \right) n^{\frac{s^{\star}-1}{s^{\star}}}  -\frac{\epsilon_0^2 n}{2^9 b^2} \right],
\end{align*}
where $\alpha =  \frac{1}{64}\epsilon_0 (b+\eta)^{-1}$. The result follows by Lemma \ref{lemma:near-net}:
\[\mathbb{P}(E^c) \leq k\left(1-   \left| S_r^{k-1} \cap B\left(e_1, \frac{\epsilon}{\sqrt{k}}\right)\right| \left|S_r^{k-1} \right|^{-1} \right)^{N_0}. \qedhere\]
\end{proof}

\section{Conclusion}\label{section:conclusion}
In this paper, we have provided an estimation algorithm for multi-index models with a coordinate-wise monotone transfer function. Our algorithm enables future work on wide-ranging applications naturally modeled as a monotone multi-index model. Promising future directions include finding an efficient method for the sparse matrix isotonic regression problem, dropping the nonnegativity assumption of the index vectors, as well as studying multi-index models with other classes of transfer functions. Studying the non-sparse setting would be interesting as well, and would present several challenging technical hurdles.
%\begin{ack}
%\end{ack}

%\textbf{Broader Impact:} This is a theoretical work that does not present any foreseeable societal consequence.

\bibliographystyle{plain}
\bibliography{MIT-Project-6,MIT_Project_4}

\begin{thebibliography}{10}

\bibitem{Beliakov2005}
Gleb Beliakov.
\newblock Monotonicity preserving approximation of multivariate scattered data.
\newblock {\em BIT Numerical Mathematics}, 45(4):653--677, 2005.

\bibitem{Dobson2018}
Annette~J. Dobson and Adrian~G. Barnett.
\newblock {\em An introduction to generalized linear models}.
\newblock {CRC Press}, 2018.

\bibitem{Dumer2007}
Ilya Dumer.
\newblock Covering spheres with spheres.
\newblock {\em Discrete \& Computational Geometry}, 38(4):665--679, 2007.

\bibitem{Gamarnik1999}
David Gamarnik.
\newblock Efficient learning of monotone concepts via quadratic optimization.
\newblock In {\em COLT}, 1999.

\bibitem{Gamarnik2019}
David Gamarnik and Julia Gaudio.
\newblock Sparse high-dimensional isotonic regression.
\newblock {\em 33rd Conference on Neural Information Processing Systems
  (NeurIPS 2019)}, 2019.

\bibitem{Bertsimas1999}
Dimitris Bertsimas~David Gamarnik and {John N.} Tsitsiklis.
\newblock Estimation of time-varying parameters in statistical models: an
  optimization approach.
\newblock {\em Machine Learning}, 35(3):225--245, 1999.

\bibitem{Hardle1993}
Wolfgang H\"ardle, Peter Hall, and Hidehiko Ichimura.
\newblock Optimal smoothing in single-index models.
\newblock {\em {The Annals of Statistics}}, pages 157--178, 1993.

\bibitem{Haussler1995}
David Haussler.
\newblock Overview of the {Probably Approximately Correct (PAC)} learning
  framework.
\newblock
  \url{https://hausslergenomics.ucsc.edu/wp-content/uploads/2017/08/smo_0.pdf},
  1995.

\bibitem{Horowitz1996}
Joel~L. Horowitz and Wolfgang H\"ardle.
\newblock Direct semiparametric estimation of single-index models with discrete
  covariates.
\newblock {\em Journal of the American Statistical Association},
  91(436):1632--1640, 1996.

\bibitem{Ichimura1991}
Hidehiko Ichimura.
\newblock Semiparametric least squares {(SLS)} and weighted {SLS} estimation of
  single-index models.
\newblock Technical report, Center for Economic Research, Department of
  Economics, University of Minnesota, 1991.

\bibitem{Kakade2011}
Sham~M. Kakade, Adam {Tauman Kalai}, Varun Kanade, and Ohad Shamir.
\newblock Efficient learning of generalized linear and single index models with
  isotonic regression.
\newblock In {\em Advances in Neural Information Processing Systems 24 (NeurIPS
  2011)}, 2011.

\bibitem{Moshkovitz2014}
Guy Moshkovitz and Asaf Shapira.
\newblock Ramsey theory, integer partitions and a new proof of the
  {Erd\H{o}s-Szekeres} theorem.
\newblock {\em Advances in Mathematics}, 262:1107--1129, 2014.

\bibitem{Kalai2009}
Adam {Tauman Kalai} and Ravi Sastry.
\newblock The {Isotron} algorithm: High-dimensional isotonic regression.
\newblock {\em Conference on Learning Theory}, 2009.

\bibitem{Yang2017}
Zhuoran Yang, Krishna Balasubramanian, Zhaoran Wang, and Han Liu.
\newblock Learning non-{Gaussian} multi-index model via a second-order
  {Stein's} method.
\newblock {\em Advances in Neural Information Processing Systems},
  30:6097--6106, 2017.

\end{thebibliography}

\newpage
\section{Deferred proofs}
\begin{proof}[Proof of Proposition \ref{proposition:scaling}]
Clearly $\overline{\mathcal{F}}_{d,k} \subseteq \mathcal{F}_{d,k}$. It remains to show that $\mathcal{F}_{d,k} \subseteq \overline{\mathcal{F}}_{d,k}$. Let $g \in \mathcal{F}_{d,k}$, where $g = f \circ \beta(I)$. We will show that $g \in \overline{\mathcal{F}}_{d,k}$. 

Suppose the $i$th column of $\beta(I)$ has norm $t < r$. Let $\overline{\beta}$ be equal to $\beta(I)$, with the $i$th column scaled by a factor of $\frac{r}{t}$, so that the $i$th column of $\overline{\beta}$ has norm $r$. Note that $\overline{\beta} = \overline{\beta}(I)$. Next, define the function $\overline{f}$ by $\overline{f}(x) = f(x_1, \dots, \frac{t}{r} x_i, \dots, x_k)$. Observe that $g = \overline{f} \circ \overline{\beta}$. We verify the monotonicity property for $\overline{f}$. Let $x \preceq y$. Then also $(x_1, \dots, \frac{t}{r} x_i, \dots, x_k) \preceq (y_1, \dots, \frac{t}{r} y_i, \dots, y_k)$, and we have 
\[\overline{f}(x) = f\left(x_1, \dots, \frac{t}{r} x_i, \dots, x_k \right) \leq f\left(y_1, \dots, \frac{t}{r} y_i, \dots, y_k \right) = \overline{f}(y).\]
It remains to show that $\overline{f}$ is $1$-Lipschitz. Let $x, y \in \mathbb{R}^k$. By the Lipschitz condition applied to $f$,
\small
\begin{align*}
&- 1 \leq  \frac{\overline{f}(x) - \overline{f}(y) }{\left\Vert \left(x_1, \dots, \frac{t}{r} x_i, \dots, x_k \right) -  \left(y_1, \dots, \frac{t}{r} y_i, \dots, y_k \right) \right\Vert_2 } \leq 1
%&\implies - \frac{\Vert (x_1, \dots, \frac{t}{r} x_i, \dots, x_k) -  (y_1, \dots, \frac{t}{r} y_i, \dots, y_k) \Vert}{\Vert x - y \Vert} \leq  \frac{\overline{f}(x) - \overline{f}(y)}{\Vert x - y\Vert} \leq \frac{\Vert (x_1, \dots, \frac{t}{r} x_i, \dots, x_k) -  (y_1, \dots, \frac{t}{r} y_i, \dots, y_k) \Vert}{\Vert x - y \Vert}\\
~~\implies~~ -1 \leq \frac{\overline{f}(x) - \overline{f}(y)}{\Vert x - y\Vert_2} \leq 1.
\end{align*}
\normalsize
This shows that $\overline{f}$ is $1$-Lipschitz. Repeating this argument for each column of $\beta(I)$, we conclude that $g \in \overline{\mathcal{F}}_{d,k}$.
\end{proof}

To prove Lemma \ref{lemma:near-net}, we use the following helper lemma. 
\begin{lemma}\label{lemma:near-net-helper}
Consider the near-net $\mathcal{R}_0$ with $N_0$ described above, and let $v$ be a fixed vector of norm $r$. With probability
\[1 -  \left(1-   \left| S_r^{k} \cap B(e_1, \delta)\right| \left|S_r^{k} \right|^{-1} \right)^{N_0},\]
%\[1 -  \left(1-  (2\epsilon)^{k-1} \left| S_r^{k-1} \right|^{-1}\right)^{N_0},\]
there exists $\overline{u} \in \mathcal{R}_0$ such that $\Vert v - \overline{u}\Vert_{2} \leq \delta$. 
\end{lemma}
\begin{proof}
Let $u$ be distributed uniformly at random on the surface of $S_r^{k-1}$. Then
\begin{align*}
\mathbb{P}\left(\Vert v - u\Vert_{2} \leq \delta \right) &= \mathbb{P}\left(u \in B(v, \delta) \right) \\
&= \left| S_r^{k-1} \cap B(v, \delta)\right| \left|S_r^{k} \right|^{-1} \\
&= \left| S_r^{k-1} \cap B(e_1, \delta)\right| \left|S_r^{k} \right|^{-1},
\end{align*}
Therefore, 
\[\mathbb{P}\left(\Vert v - u\Vert_{2} > \delta \right) = 1-   \left| S_r^{k} \cap B(e_1, \delta)\right| \left|S_r^{k} \right|^{-1}.\]
We conclude that the probability that there exists an element of $\mathcal{R}_0$ that is $\delta$-close to $v$ is equal to
\[1 -  \left(1-   \left| S_r^{k} \cap B(e_1, \delta)\right| \left|S_r^{k} \right|^{-1} \right)^{N_0}. \qedhere\]
\end{proof}
\begin{proof}[Proof of Lemma \ref{lemma:near-net}]
Let $E$ be the event that there exists $\overline{R} \in \mathcal{R}$ such that $\Vert M - \overline{R}\Vert_{F} \leq \epsilon$. We need to lower bound the probability of the event $E$.
Let $\{M_i\}_{i=1}^k$ denote the columns of $M$. For $i \in [k]$, let $E_i$ be the event that there exists $\overline{u}_i \in \mathcal{R}_0$ such that $\Vert M_i - \overline{u}_i \Vert_2 \leq \frac{\epsilon}{\sqrt{k}}$. We claim that $\mathbb{P}\left(\cap_{i=1}^k E_i\right) \leq \mathbb{P}(E)$. Indeed, suppose that the event $E_i$ occurs for each $i$. Let $\overline{R} \in \mathcal{R}$ be the matrix with columns $\{\overline{u}_i\}_{i=1}^k$. Then
\begin{align*}
\Vert M - \overline{R} \Vert_F^2 = \sum_{i=1}^k \Vert M_i - \overline{u}_i \Vert_2^2 &\leq k \left(\frac{\epsilon}{\sqrt{k}}\right)^2 = \epsilon^2,
\end{align*}
and so $\Vert M - \overline{R} \Vert_F \leq \epsilon$. This shows that $\mathbb{P}\left(\cap_{i=1}^k E_i\right) \leq \mathbb{P}(E)$. We also have
\[\mathbb{P}(E^c) \leq \mathbb{P}\left(\cup_{i=1}^k E_i^c\right) \leq \sum_{i=1}^k \mathbb{P}(E_i^c),\]
so that $\mathbb{P}(E) \geq 1 - \sum_{i=1}^k \mathbb{P}(E_i^c).$

For each $i \in [k]$, Lemma \ref{lemma:near-net-helper} shows that there exists $\overline{u}_i \in \mathcal{R}_0$ such that $\Vert M_i - \overline{u}_i \Vert_2 \leq \frac{\epsilon}{\sqrt{k}}$, with probability 
\[1 -  \left(1-   \left| S_r^{k} \cap B\left(e_1, \frac{\epsilon}{\sqrt{k}}\right)\right| \left|S_r^{k} \right|^{-1} \right)^{N_0}.\] 
Therefore,
\[\mathbb{P}(E_i^c) \leq \left(1-   \left| S_r^{k} \cap B\left(e_1, \frac{\epsilon}{\sqrt{k}}\right)\right| \left|S_r^{k} \right|^{-1} \right)^{N_0}. \qedhere\]
\end{proof}

\begin{proof}[Proof of Proposition \ref{proposition:optimal}]
We first show that the constraints enforce the monotonicity requirement. Let $I = \{i : v_i =1\}$. Consider two samples $(X_i ,Y_i)$ and $(X_j, Y_j)$. The monotonicity requirement is
\[ (M(I))^T X_i) \preceq (M(I))^T X_j) \implies F_i \leq F_j.\]
The contrapositive of this statement is
\begin{align}
F_i > F_j \implies \exists p \in [k] : ((M(I))^T X_i)_p > ((M(I))^T X_j)_p. \label{eq:contrapositive}
\end{align}
The optimization encodes the contrapositive statement, as follows. There are two cases: either $F_i > F_j$ or $F_i \leq F_j$. We must ensure that if $F_i > F_j$ holds, then the implication in \eqref{eq:contrapositive} is satisfied. We must also verify that no additional constraints are introduced if $F_i \leq F_j$.

Suppose $F_i > F_j$. Then $z_{ij} = 1$ by Constraint \eqref{eq:ordering}. By Constraint \eqref{eq:auxiliary}, at least one of the $q_{ijp}$ variables must be equal to $1$. Then by Constraints \eqref{eq:monotonicity}, we have
\[\sum_{l=1}^d v_l M_{lp} (X_{il} - X_{jl}) > 0 \iff ((M(I))^T X_i)_p > ((M(I))^T X_j)_p,\]
for at least one $p \in [k]$, due to the choice of $\mu$. Next suppose $F_i \leq F_j$. Then $z_{ij}$ is free to equal zero, and all the $q_{ijp}$ values may be set to zero as well. By the choice of $B$, Constraint \eqref{eq:monotonicity} is then non-binding.

The objective minimizes the loss on the samples. Finally, we claim that the choice of $\hat{f}_n$ is a monotone interpolation. First, $\hat{f}_n(M(I_n)^T X_i )= F_i$, so that $\hat{f}_n$ interpolates. Next, observe that $x \preceq y \implies \hat{f}_n(x) \leq \hat{f}_n(y)$. Also, $\hat{f}_n : \mathbb{R} \to [0,b]$, by construction.
\end{proof}

\begin{proof}[Proof of Lemma \ref{lemma:net-sensitivity-main}]
Fix $(x,y)$.
%such that $F(x,y) > 0$. Then $\left| (f^{\star} \circ Q^{\star} R^{\star})(x) - y \right| \leq \eta$ almost surely. 
We have
\begin{align}
&L(x,y, f^{\star} \circ (TR)^+(I^{\star})) dF(x,y)  - L(x,y, f^{\star} \circ Q^{\star} R^{\star}) \nonumber \\
&=  \left(y - (f^{\star} \circ (TR)^+(I^{\star}))(x)\right)^2 -  \left(y - (f^{\star} \circ Q^{\star} R^{\star})(x)\right)^2  \nonumber \\
&= \left((f^{\star} \circ (TR)^+(I^{\star}))(x) - (f^{\star} \circ Q^{\star} R^{\star})(x) \right) \left((f^{\star} \circ (TR)^+(I^{\star}))(x) + (f^{\star} \circ Q^{\star} R^{\star})(x) -2y \right)  \nonumber \\
&\leq \left|(f^{\star} \circ Q^{\star} R^{\star})(x) - (f^{\star} \circ (TR)^+(I^{\star}))(x) \right| \cdot \left| (f^{\star} \circ Q^{\star} R^{\star})(x) + (f^{\star} \circ (TR)^+(I^{\star}))(x) - 2y\right| \label{eq:product}
\end{align}
Since $f^{\star} \in \mathcal{L}_1(b)$, it holds that
\begin{align*}
\left|(f^{\star} \circ Q^{\star} R^{\star})(x) - (f^{\star} \circ TR^+(I^{\star}))(x) \right| &\leq \left \Vert (Q^{\star} R^{\star})^Tx - ((TR)^+(I^{\star}))^T  x \right\Vert_2 \\
&= \left \Vert \left(Q^{\star} R^{\star}  - (TR)^+(I^{\star}) \right)^T x \right\Vert_2.
\end{align*}
%Using the fact that $Q^{\star} R^{\star}$ is a nonnegative matrix, along with Lemma \ref{lemma:net-density-bound-main}, we have
%\[\Vert Q^{\star}R^{\star} - (TR)^+ \Vert_{\infty} \leq \Vert Q^{\star}R^{\star} - TR \Vert_{\infty} \leq \epsilon_1k + \epsilon_2 r. \]
For the second factor in the bound \eqref{eq:product}, we have
\begin{align*}
&\left| (f^{\star} \circ Q^{\star} R^{\star})(x) + (f^{\star} \circ TR^+(I^{\star}))(x) - 2y\right|  \\
 &\leq 2 \left| (f^{\star} \circ Q^{\star} R^{\star})(x) - y \right| + \left|(f^{\star} \circ TR^+(I^{\star}))(x) - (f^{\star} \circ Q^{\star}R^{\star})(x)\right|\\
 &\leq 2 \left| (f^{\star} \circ Q^{\star} R^{\star})(x) - y \right| + \left \Vert \left(Q^{\star} R^{\star}  - (TR)^+(I^{\star}) \right)^T x \right\Vert_2.
 \end{align*}
%Let $\Delta = \epsilon_1 k + \epsilon_2 r$. 
Let $A = Q^{\star} R^{\star}  - (TR)^+(I^{\star})$. Substituting into \eqref{eq:product}, we have
\begin{align}
&\int L(x,y, f^{\star} \circ (TR)^+) dF(x,y) - \int L(x,y, f^{\star} \circ Q^{\star} R^{\star}) dF(x,y) \nonumber \\
&\leq \int \left \Vert A^T x \right \Vert_2 \left(2 \left| (f^{\star} \circ Q^{\star} R^{\star})(x) - y \right| +\left \Vert A^T x \right \Vert_2 \right) dF(x,y) \nonumber \\
&\leq \int \left \Vert A^T x \right \Vert_2 \left(2\eta +\left \Vert A^T x \right \Vert_2 \right) dF(x,y) \nonumber  \\
&= \int \left \Vert A^T x \right \Vert_2 \left(2\eta +\left \Vert A^T x \right \Vert_2 \right) dF_X(x) \nonumber  \\
&=  2 \eta \mathbb{E}\left[ \left \Vert A^T X \right \Vert_2 \right]+  \mathbb{E} \left[ \left \Vert A^T X \right \Vert_2^2 \right] \nonumber \\
&\leq  2 \eta \sqrt{\mathbb{E}\left[ \left \Vert A^T x \right \Vert_2^2 \right]}+  \mathbb{E} \left[ \left \Vert A^T x \right \Vert_2^2 \right], \label{eq:norm-expression}
\end{align}
where the last inequality follows from Jensen's inequality. We now evaluate the expectation.
\begin{align*}
\mathbb{E}\left[ \left \Vert A^T X \right \Vert_2^2 \right] &= \mathbb{E} \left[ \sum_{j=1}^k \left(\sum_{i = 1}^d A_{ij} X_i \right)^2 \right]
%&=  \sum_{j=1}^k \mathbb{E} \left[ \left(\sum_{i = 1}^d A_{ij} x_i \right)^2 \right]\\
%&=  \sum_{j=1}^k \mathbb{E} \left[ \sum_{i = 1}^d \sum_{l = 1}^d A_{ij} A_{lj} X_i X_l \right]\\
=  \sum_{j=1}^k \sum_{i = 1}^d \sum_{l = 1}^d A_{ij} A_{lj} \mathbb{E} \left[   X_i X_l \right].
\end{align*}
Recall that the coordinates of the random variable $X$ are independent and have zero mean. Therefore,
\begin{align*}
\mathbb{E}\left[ \left \Vert A^T X \right \Vert_2^2 \right] =  \sum_{j=1}^k \sum_{i = 1}^d A_{ij}^2 \mathbb{E} \left[   X_i^2 \right] \leq C^2  \sum_{j=1}^k \sum_{i = 1}^d A_{ij}^2 = C^2 \Vert A \Vert_F^2.
 \end{align*}
Using the fact that $Q^{\star} R^{\star} \geq 0$ and $Q^{\star} R^{\star} = (Q^{\star} R^{\star})(I^{\star})$, we have
\begin{align*}
 \Vert A \Vert_F &=  \Vert Q^{\star} R^{\star}  - (TR)^+(I^{\star}) \Vert_F\\
 &\leq  \Vert Q^{\star} R^{\star}  - TR \Vert_F\\
 &\leq  \Vert Q^{\star} R^{\star}  - TPR^{\star} \Vert_F + \Vert TPR^{\star}  - TR\Vert_F\\
 &=  \Vert (Q^{\star}- TP) R^{\star}  \Vert_F + \Vert T(PR^{\star}  - R)\Vert_F\\
 &\leq  \Vert Q^{\star}- TP \Vert_F \Vert R^{\star}\Vert_F  + \Vert T\Vert_F \Vert PR^{\star}  - R\Vert_F\\
  &\leq \epsilon_2 \Vert R^{\star}\Vert_F  + \epsilon_1 \Vert T\Vert_F \\
  &= \epsilon_2 \sqrt{k r^2} + \epsilon_1 \sqrt{k}\\
  &= \sqrt{k} \left(\epsilon_1 + \epsilon_2 r \right).
\end{align*}
Therefore, 
\[ \mathbb{E}\left[ \left \Vert A^T X \right \Vert_2^2 \right]  \leq C^2 \Vert A \Vert_F^2 \leq C^2k \left( \epsilon_1 + \epsilon_2 r \right)^2.\] 
Substituting into \eqref{eq:norm-expression}, we conclude
\begin{align*}
&\int L(x,y, f^{\star} \circ (TR)^+) dF(x,y) - \int L(x,y, f^{\star} \circ Q^{\star} R^{\star}) dF(x,y) \\
&\leq 2 \eta C \sqrt{k} \left(\epsilon_1 + \epsilon_2 r \right) + C^2 k \left(\epsilon_1 + \epsilon_2 r \right)^2 = z(\epsilon_1, \epsilon_2, C). \qedhere
\end{align*}
\end{proof}

\begin{proof}[Proof of Lemma \ref{lemma:norm-integral}]
A similar result appears in \cite{Bertsimas1999}. We include the proof for completeness.
\small
\begin{align*}
&\int L(x,y, g) dF(x,y) \\
&= \int \left(g(x) - y\right)^2 dF(x,y)\\
&=  \int  \left(g(x) - (f^{\star} \circ Q^{\star} R^{\star})(x) + (f^{\star} \circ Q^{\star} R^{\star})(x) - y\right)^2 dF(x,y)\\
&=  \int  \left(g(x) - (f^{\star} \circ Q^{\star} R^{\star})(x)\right)^2 + \left((f^{\star} \circ Q^{\star} R^{\star})(x) - y\right)^2 + 2 \left(g(x) - (f^{\star} \circ Q^{\star} R^{\star})(x)\right)\left((f^{\star} \circ Q^{\star} R^{\star})(x) - y\right)  dF(x,y)\\
&=  \Vert g - f^{\star} \circ Q^{\star} R^{\star}\Vert_2^2  + \int L(x,y, f^{\star}\circ Q^{\star} R^{\star}) dF(x,y) + 2 \mathbb{E}\left[ \left(g(X) - (f^{\star} \circ Q^{\star} R^{\star})(X) \right) \left((f^{\star} \circ Q^{\star} R^{\star})(X) - Y\right)\right].
\end{align*}
\normalsize
Since $\mathbb{E}[Y|X] = (f^{\star} \circ Q^{\star} R^{\star})(x)$, we have
\[\mathbb{E}\left[ \left(g(X) - (f^{\star} \circ Q^{\star} R^{\star})(X) \right) \left((f^{\star} \circ Q^{\star} R^{\star})(X) - Y\right)\right] = 0.\]
Rearranging completes the proof.
\end{proof}

\begin{proof}[Proof of Lemma \ref{lemma:sample-accuracy-main-helper}]
Recalling that the range of any $h \in \mathcal{G}(T, \mathcal{R})$ is contained in $[0,b]$, the statement follows from Corollary 1 (pp. 45) of \cite{Haussler1995}.
\end{proof}

%%% BIG PROOF HERE

We now work towards a proof of Lemma \ref{lemma:net-bound}. Recall $\alpha = \alpha(\epsilon) = \frac{1}{2}\epsilon (b+\eta)^{-1}$ and let $S = \{0, \alpha, 2 \alpha, \dots,  \left(\left\lceil \frac{b}{\alpha} \right \rceil - 1\right) \alpha\}$ be a discretization of the range $[0,b]$. For $f \in \mathcal{C}(b)$, let 
\begin{align*}
g_f(x) = \max \{q \in S: q \leq f(x)\}.
\end{align*}
In other words, the function $g_f$ is formed by rounding each value down to the nearest increment in $S$. Fix $I \subset [d]$ with $|I| = s^{\star}$. Recall the definition $\mathcal{G}(T, R, I) = \{f \circ (TR)^+(I) : f \in \mathcal{C}(b)\}$, and let $\mathcal{H}(T, R, I) \triangleq \{g_f \circ (TR)^+(I) : f \in \mathcal{C}(b)\}$. Proposition \ref{proposition:net-proof} will show that the set $\mathbf{L}_{\mathcal{H}(Q, R, I)}((x_1,y_1), \dots, (x_n, y_n))$ is an $\epsilon$-net for the set $\mathbf{L}_{\mathcal{G}(Q, R, I)}((x_1,y_1), \dots, (x_n, y_n)) $. Next, we relate the cardinality of $\mathbf{L}_{\mathcal{H}(Q,R, I)}((x_1,y_1), \dots, (x_, y_n))$ to a \emph{labeling number}.
\begin{definition}[Labeling Number \cite{Gamarnik1999}]
For a sequence of points $x_1, \dots, x_n \in \mathbb{R}^k$ and a positive integer $m$, the labeling number $\Lambda(x_1, \dots, x_n; m)$ is the number of functions $\phi: \{x_1, \dots, x_n\} \to \{1, 2, \dots, m\}$ such that $\phi(x_i) \leq \phi(x_j)$ whenever $x_i \preceq x_j$ for $i, j \in \{1, \dots, n\}$.
\end{definition}
Let $M = (QR)^+$. Observe that the cardinality of $\mathbf{L}_{\mathcal{H}(Q,R, I)}((x_1,y_1), \dots, (x_, y_n))$ is upper-bounded by the labeling number of the set $\{M(I)^Tx_1, \dots, M(I)^T x_n\}$ with $\left( \left \lceil \frac{b}{\alpha} \right \rceil - 1 \right)$ labels. Therefore, the value of $N_{\mathcal{G}(Q,R,I)}\left(\epsilon, n\right)$ is upper-bounded by the \emph{expected} labeling number of the set $\{M(I)^T X_1, \dots M(I)^T X_n\}$ with $\left( \left \lceil \frac{b}{\alpha} \right \rceil - 1 \right)$ labels.

We therefore need to determine an upper bound on the expected labeling number of the set $\{M(I)^T X_1, \dots M(I)^T X_n\}$ with $\left( \left \lceil \frac{b}{\alpha} \right \rceil - 1 \right)$ labels. Let $\overline{x}(I)$ be the vector formed from the entries of $x$ that are indexed by the set $I$. We will first show that the labeling number of the set $\{M(I)^T x_1, \dots M(I)^T x_n\}$ is upper-bounded by the labeling number of the set $\{\overline{x_1}(I), \dots, \overline{x_n}(I)\}$ with the same number of labels. Observe that the points $\{\overline{x_i}(I)\}_{i=1}^n$ have dimension $s^{\star}$. We will then analyze the expected labeling number of a sequence of random variables $\{W_1, \dots, W_n\}$ that are of dimension $s^{\star}$.

The following proposition will be used to establish the net property.
\begin{proposition}\label{proposition:net-proof}
Let $T \in \mathcal{O}_{d,k}$, $R \in \overline{\mathcal{M}}_{k,k}(r)$, and $I \subset [d]$. Fix $(f \circ (TR)^+(I)) \in \mathcal{G}(T, R, I)$ and the accompanying $(g_f \circ (TR)^+) \in \mathcal{H}(T, R, I)$. Let $x \in \mathcal{X}$ and $y \in \mathcal{Y}$. Then
\begin{align*}
\left| L(x, y, f \circ (TR)^+(I)) - L(x, y, g_f \circ (TR)^+(I)) \right| \leq \epsilon.
\end{align*}
\end{proposition}
\begin{proof}
Let $M = (TR)^+$.
\begin{align*}
&\left| L(x, y, f \circ (TR)^+(I)) - L(x, y, g_f \circ (TR)^+(I)) \right| \\
&=\left| L(x, y, f \circ M(I)) - L(x, y, g_f \circ M(I)) \right| \\
&= \left| \left( (f \circ M(I))(x) - y\right)^2 - \left( (g_f \circ M(I))(x) - y\right)^2 \right|\\
&= \left| (f \circ M(I))^2(x) - (g_f \circ M(I))^2(x) -2y \left((f \circ M(I))(x) - (g_f \circ M(I))(x) \right) \right|\\
&= \left| \left((f \circ M(I))(x) - (g_f \circ M(I))(x) \right) \left((f \circ M(I))(x) + (g_f \circ M(I))(x) -2y\right)\right|\\
&= \left| (f \circ M(I))(x) - (g_f \circ M(I))(x) \right| \cdot \left|(f \circ M(I))(x) + (g_f \circ M(I))(x) -2y\right|\\
&\leq \alpha \cdot 2(b + \eta)\\
&= \epsilon. \qedhere
\end{align*}
\end{proof}

Let $x_1, \dots, x_n \in \mathbb{R}^d$. The following result will allow us to relate the binary labeling number of the set $\{M(I) x_1, \dots M(I) x_n\}$ to the binary labeling number of the set $\{x_1(I), \dots, x_n(I)\}$.
\begin{proposition}\label{proposition:labelling-contraction} 
Let $A$ be a $d \times k$ matrix with nonnegative entries. Let $x_1, \dots, x_n \in \mathbb{R}^d$. Then for $m \geq 1$,
\[\Lambda \left( A^Tx_1, \dots, A^Tx_n; m\right) \leq \Lambda \left( x_1, \dots, x_n; m\right).\]
\end{proposition}
\begin{proof}
Suppose $x_i \preceq x_j$. Then also $Ax_i \preceq Ax_j$. Therefore any labeling that is feasible for the points $\{A^T x_1, \dots, A^T x_n\}$ is also feasible for the points $\{x_1, \dots, x_n\}$.
\end{proof}

We will now analyze the expected labeling number of a set of random variables $\{W_1, \dots, W_n\}$. The concept of an \emph{integer partition} is required for the labeling number analysis.
\begin{definition}[Integer Partition (as stated in \cite{Gamarnik2019})]
An integer partition of dimension $(k-1)$ with values in $\{0, 1, \dots, t\}$ is a collection of values $A_{i_1, i_2, \dots, i_{k-1}} \in \{0, 1, \dots, t\}$ where $i_l \in \{1, \dots m\}$ and $A_{i_1, i_2, \dots, i_{d-1}} \leq A_{j_1, j_2, \dots, j_{k-1}}$ whenever $i_l \leq j_l$ for all $l \in \{1, \dots, k-1\}$. The set of integer partitions of dimension $(k-1)$ with values in $\{0,1, \dots, t\}$ is denoted by $P\left([t]^k\right)$.
\end{definition}
For an illustration of a partition with $k=2$, see Figure 3 in \cite{Gamarnik2019}. The following result provides a bound on the expected labeling number of a set of points with certain distribution assumptions. It will be used to bound the expected labeling number of the set $\{X_1, \dots, X_n\}$.
\begin{lemma}\label{lemma:labeling-number-bound}
Let $m \in \mathbb{N}$ and $B > 0$. Let $W \in \mathbb{R}^d$ be a random variable with support contained in the set $[-B,B]^d$. Suppose that the density $f_W(w)$ is upper-bounded by $D$. Let $W_1, \dots, W_n$ be independent samples with distribution $f_W$. Then 
\begin{align*}
\mathbb{E}\left[\Lambda(W_1, \dots, W_n; m)\right] &\leq \exp \left[ \left( 2\log(2)(m-1) + D2^{m + 2d -1} B^d \right) n^{\frac{d-1}{d}} \right].
\end{align*}
\end{lemma}
\begin{proof}
Note that since the labeling number is translation-invariant, the same result applies to a random variable $W$ with support contained in $[0,2B]^d$. \cite{Gamarnik2019} considered the case $B = \frac{1}{2}$ and $D = 1$. We now adapt their proof. By Proposition 5 in \cite{Gamarnik2019},
\begin{align*}
\Lambda(w_1, \dots, w_n; m) \leq \left( \Lambda(w_1, \dots, w_n; 2)\right)^{m-1}
\end{align*}
for any $w_1, \dots, w_n \in \mathbb{R}^d$. For clarity of notation, write $\Lambda(w_1, \dots, w_n)$ in place of $\Lambda(w_1, \dots, w_n; 2)$. We therefore have
\begin{align*}
\mathbb{E}\left[\Lambda(W_1, \dots, W_n; m)\right] &\leq \mathbb{E}\left[\left(\Lambda(W_1, \dots, W_n)\right)^{m-1}\right].
\end{align*}
Let $t \in \mathbb{N}$. When $B = \frac{1}{2}$ and $D = 1$, we have by Lemma 5 of \cite{Gamarnik2019}
\begin{align*}
\mathbb{E}\left[\left(\Lambda(W_1, \dots, W_n)\right)^{m-1}\right] &\leq \left| P([t]^d)\right|^{m-1} \mathbb{E}\left[ 2^{(m-1)N}\right],
\end{align*}
where $N \sim \text{Binom}\left(n, \frac{t^d - (t-1)^d}{t^d}\right)$. The value $ \frac{t^d - (t-1)^d}{t^d}$ is the probability that a uniform random variable in $[0,1]^d$ falls in one of $t^d - (t-1)^d$ cubes out of $t^d$ cubes that partition $[0,1]^d$. To adapt the proof, we instead partition $[-B,B]^d$ into $t^d$ cubes. Each cube therefore has volume $\left(\frac{2B}{t}\right)^d$, and density upper-bounded by $D\left(\frac{2B}{t}\right)^d$. We conclude that
\begin{align}
\mathbb{E}\left[\left(\Lambda(W_1, \dots, W_n)\right)^{m-1}\right] &\leq \left| P([t]^d)\right|^{m-1} \mathbb{E}\left[ 2^{(m-1)N'}\right], \label{eq:intermediate-step}
\end{align}
where $N' \sim \text{Binom}\left(n, D\left(\frac{2B}{t}\right)^d\left(t^d - (t-1)^d\right)\right)$. Let $p = D\left(\frac{2B}{t}\right)^d\left(t^d - (t-1)^d\right)$.

\cite{Moshkovitz2014} showed that
\[\left|P([t]^d) \right| \leq \binom{2t}{t}^{t^{d-2}}.\]
We have
\begin{align*}
\mathbb{E}\left[ 2^{(m-1)N'}\right] = \mathbb{E}\left[ e^{\log(2)(m-1)N')}\right] = M_{N'}(\log(2)(m-1)),
\end{align*}
where $M_{N'}$ is the moment-generating function of the random variable $N'$. It holds that
\begin{align*}
M_{N'}(\theta) &= (1 - p + p e^{\theta})^n.
\end{align*}
Substituting into \eqref{eq:intermediate-step},
\begin{align*}
\mathbb{E}\left[\left(\Lambda(W_1, \dots, W_n)\right)^{m-1}\right] &\leq \left( \binom{2t}{t}^{t^{d-2}}\right)^{m-1} \left(1 - p + p e^{\log(2)(m-1)}\right)^n\\
&= \binom{2t}{t}^{(m-1)t^{d-2}} \left(1 - p + p 2^{(m-1)}\right)^n\\
&\leq 2^{2t(m-1)t^{d-2}} \left(1 + p 2^{(m-1)}\right)^n\\
&= \exp \left[2\log(2)(m-1)t^{d-1}  + n \log\left(1 + p 2^{(m-1)}\right) \right].
\end{align*}
Using the fact that $\log(1+x) \leq x$, we have
\begin{align*}
\mathbb{E}\left[\left(\Lambda(W_1, \dots, W_n)\right)^{m-1}\right] &\leq \exp \left[2\log(2)(m-1)t^{d-1} + n  p 2^{(m-1)} \right]\\
&= \exp \left[2\log(2)(m-1)t^{d-1} + D2^{(m-1)} \left(\frac{2B}{t}\right)^d\left(t^d - (t-1)^d\right) n \right].
\end{align*}
By the Binomial Theorem,
\begin{align*}
t^d - (t-1)^d &= t^d - \sum_{i=0}^d \binom{d}{i} t^{d-i} (-1)^i\\
& = \sum_{i=1}^d \binom{d}{i} t^{d-i} (-1)^{i+1} \\
& \leq \sum_{i=1}^d \binom{k}{i} \max_{i \in \{1, \dots, d\}} t^{d-i} (-1)^{i+1} \\
&= (2^d -1) t^{d-1}.
\end{align*}
Substituting,
\begin{align*}
\mathbb{E}\left[\left(\Lambda(W_1, \dots, W_n)\right)^{m-1}\right] &\leq \exp \left[2\log(2)(m-1)t^{d-1} + D2^{(m-1)} \left(\frac{2B}{t}\right)^d (2^d -1) t^{d-1} n \right]\\
&= \exp \left[2\log(2)(m-1)t^{d-1} + D2^{(m-1)} \left(2B\right)^d (2^d -1) t^{-1} n \right]\\
&\leq \exp \left[2\log(2)(m-1)t^{d-1} + D2^{m + 2d -1} B^d t^{-1} n \right]
\end{align*}

Let $t = n^{\frac{1}{d}}$. Substituting,
\begin{align*}
\mathbb{E}\left[\left(\Lambda(W_1, \dots, W_n)\right)^{m-1}\right] &\leq  \exp \left[2\log(2)(m-1)n^{\frac{d-1}{d}} + D2^{m + 2d -1} B^d n^{\frac{d-1}{d}} \right]\\
&= \exp \left[ \left( 2\log(2)(m-1) + D2^{m + 2d -1} B^d \right) n^{\frac{d-1}{d}} \right]. \qedhere 
\end{align*}
\end{proof}

We now prove Lemma \ref{lemma:net-bound}.
\begin{proof}[Proof of Lemma \ref{lemma:net-bound}]
Recall the definitions $\mathcal{G}(T, R, I) = \{f \circ (TR)^+(I) : f \in \mathcal{C}(b)\}$ and $\mathcal{H}(T, R, I) = \{g_f \circ (TR)^+(I) : f \in \mathcal{C}(b)\}$ for $R \in \mathcal{R}$. We have
\begin{align*}
N_{\mathcal{G}(T, \mathcal{R})}\left(\epsilon, n\right) \leq \sum_{R \in \mathcal{R}} \sum_{I \subset [d] : |I| = s^{\star}} N_{\mathcal{G}(T, R, I)}\left(\epsilon, n\right) \leq \binom{d}{s^{\star}}  |\mathcal{R}| \max_{R \in \mathcal{R}, I \subset [d] : |I| = s^{\star}} N_{\mathcal{G}(T,R,I)}\left(\epsilon, n\right).
\end{align*}
%\begin{align*}
%N_{\mathcal{G}}\left(\epsilon, n\right) \leq \sum_{R \in \mathcal{R}} N_{\mathcal{G}(R)}\left(\epsilon, n\right) \leq |\mathcal{R}| \max_{R \in \mathcal{R}} N_{\mathcal{G}(R)}\left(\epsilon, n\right).
%\end{align*}
Consider an arbitrary $R \in \mathcal{R}$ and $I \subset [d]$ with $|I| = s^{\star}$. By Proposition \ref{proposition:net-proof}, the set $\mathbf{L}_{\mathcal{H}(T,R,I)}((x_1,y_1), \dots, (x_n, y_n))$ is an $\epsilon$-net for the set $\mathbf{L}_{\mathcal{G}(T,R,I)}((x_1,y_1), \dots, (x_n, y_n)) $. Let $M = (TR)^+$. Observe that the cardinality of $\mathbf{L}_{\mathcal{H}(T,R,I)}((x_1,y_1), \dots, (x_, y_n))$ is upper-bounded by the labeling number of the set $\{(M(I))^Tx_1, \dots, (M(I))^T x_n\}$ with $\left( \left \lceil \frac{b}{\alpha} \right \rceil - 1 \right)$ labels. Recall the definition of $\overline{x}(I)$, and similarly let $\overline{M}(I)$ be the matrix formed from the rows of $M$ that are indexed by the set $I$. 
For $x \in \mathbb{R}^d$, it holds that
\[ (M(I))^T x = (\overline{M}(I))^T (\overline{x}(I)).\]
Note that $\overline{x}(I)$ is an $s^{\star}$-dimensional vector. By Proposition \ref{proposition:labelling-contraction},
\[ \Lambda \left( (M(I))^Tx_1, \dots, (M(I))^T x_n; \left( \left \lceil \frac{b}{\alpha} \right \rceil - 1 \right) \right) \leq \Lambda \left( \overline{x_1}(I), \dots, \overline{x_n}(I); \left( \left \lceil \frac{b}{\alpha} \right \rceil - 1 \right) \right).\]
Therefore, the value of $N_{\mathcal{G}(T,R,I)}\left(\epsilon, n\right)$ is upper-bounded by the \emph{expected} labeling number of the set $\{\overline{X_1}(I), \dots, \overline{X_n}(I)\}$ with $\left( \left \lceil \frac{b}{\alpha} \right \rceil - 1 \right)$ labels. Applying Lemma \ref{lemma:labeling-number-bound} (setting $d = s^{\star}$, $m = \left( \left \lceil \frac{b}{\alpha} \right \rceil - 1 \right)$, $B = C$ and $D = (p^{\star })^{s^{\star}}$), 
we have
\begin{align*}
N_{\mathcal{G}(T,R,I)}\left(\epsilon, n\right)& \leq  \exp \left[ \left( 2\log(2)\left(\left \lceil \frac{b}{\alpha} \right \rceil - 2 \right) +  2^{\left \lceil \frac{b}{\alpha} \right \rceil - 2 + 2s^{\star}} (p^{\star}C)^{s^\star} \right) n^{\frac{s^{\star}-1}{s^{\star}}} \right] \\
& \leq  \exp \left[ \left( \frac{2 \log(2)b}{\alpha} +2^{\frac{b}{\alpha}} (4p^{\star}C)^{s^{\star}} \right) n^{\frac{s^{\star}-1}{s^{\star}}} \right].
\end{align*}
We conclude that
\[N_{\mathcal{G}(T, \mathcal{R})}\left(\epsilon, n\right)  \leq  \binom{d}{s^{\star}} |\mathcal{R}| \exp \left[ \left( \frac{2 \log(2)b}{\alpha} +2^{\frac{b}{\alpha}} (4p^{\star}C)^{s^{\star}} \right) n^{\frac{s^{\star}-1}{s^{\star}}} \right].\]
Recalling that $|\mathcal{R}| = N_0^k$ concludes the proof.
\end{proof}
\begin{remark}
In the proof of Lemma \ref{lemma:net-bound}, we have taken advantage of the fact that $s^{\star}$ is a constant. Another approach would be to analyze the labeling number directly in $\mathbb{R}^k$, since $k$ is also a constant. However, there is a technical hurdle to overcome. The grid approach in Lemma \ref{lemma:labeling-number-bound} works well when the distribution of the random variable is not too concentrated. Without good control over the induced distribution of $M(I)^T X$, it would be difficult to carry out a similar argument.
\end{remark}

\begin{proof}[Proof of Theorem \ref{theorem:sample-main-result}]
We show that Algorithm \ref{alg:full} achieves the desired statistical guarantee, using the bound in Theorem \ref{theorem:main-result}.
We choose \[N_0 = \left \lceil \frac{\log\left(\frac{\epsilon}{3k}\right)}{\log \left(1-   \left| S_r^{k-1} \cap B\left(e_1, \frac{\delta}{\sqrt{k}}\right)\right| \left|S_r^{k-1} \right|^{-1} \right)} \right \rceil.\]
The choice of $N_0$ leads to
\[k\left(1-   \left| S_r^{k-1} \cap B\left(e_1, \frac{\delta}{\sqrt{k}}\right)\right| \left|S_r^{k-1} \right|^{-1} \right)^{N_0} \leq \frac{\epsilon}{3}.\]
Similarly, $\frac{1}{d^2} \leq \frac{\epsilon}{3}$ by the assumption on $d$.

Before bounding the last term, we need to control the value of $\epsilon_0$. Observe that the function $z(\epsilon_1, \epsilon_2, C)$ is decreasing in both arguments. Therefore, by setting 
\[\frac{1}{\rho_0} 4 \sqrt{2} s^{\star} \lambda \leq \delta ,\]
we ensure that $\epsilon_0 \geq \frac{\epsilon}{2}$. Substituting in the value of $\lambda$ and solving,
\begin{align*}
&\frac{1}{\rho_0} 4 \sqrt{2} s^{\star} \cdot 10 \sqrt{\theta \frac{\log(d)}{n}} \leq \delta\\
&n \geq \frac{3200  s^{\star 2} \theta \log(d)}{\rho_0^2 \delta^2} .
\end{align*}
We now bound the last term:
\begin{align}
&4\binom{d}{s^{\star}} N_0^k \exp \left[ \left( \frac{2 \log(2)b}{\alpha} +2^{\frac{b}{\alpha}} (4p^{\star}C)^{s^{\star}} \right) n^{\frac{s^{\star}-1}{s^{\star}}}  -\frac{\epsilon_0^2 n}{2^9 b^2} \right] \nonumber\\
&\leq \exp \left[\log(4) + s^{\star} \log(d) + \left( \frac{2 \log(2)b}{\alpha} +2^{\frac{b}{\alpha}} (4p^{\star}C)^{s^{\star}} \right) n^{\frac{s^{\star}-1}{s^{\star}}}  -\frac{\epsilon_0^2 n}{2^9 b^2} \right]. \label{eq:bound}
\end{align}

Recall that $\alpha =  \frac{1}{64}\epsilon_0 (b+\eta)^{-1}$. For $n \geq 3200  s^{\star 2} \theta \log(d)\rho_0^{-2} \delta^{-2}$, we have $\alpha \geq  \frac{1}{128}\epsilon (b+\eta)^{-1}$. We see that there exists $t = t(N_0, C, b, s^{\star}, p^{\star}, k, \eta)$ such that if $n \geq \max \left\{t,   3200  s^{\star 2} \theta \log(d)\rho_0^{-2} \delta^{-2}\right \}$, then 
\[\log(4) + \left( \frac{2 \log(2)b}{\alpha} +2^{\frac{b}{\alpha}} (4p^{\star}C)^{s^{\star}} \right) n^{\frac{s^{\star}-1}{s^{\star}}}  -\frac{\epsilon_0^2 n}{2^9 b^2} \leq \frac{\epsilon_0^2 n}{2^{10} b^2}.\]
For such $n$, we then bound \eqref{eq:bound} by $\exp \left[s^{\star} \log(d) -\frac{\epsilon_0^2 n}{2^{10} b^2} \right]$. Setting this quantity to be less than $\frac{\epsilon}{3}$, we obtain
\begin{align*}
&\exp \left[s^{\star} \log(d) -\frac{\epsilon_0^2 n}{2^{10} b^2} \right] \leq \frac{\epsilon}{3}\\
&s^{\star} \log(d) -\frac{\epsilon_0^2 n}{2^{10} b^2} \leq\log \left( \frac{\epsilon}{3}\right)\\
&n \geq \frac{2^{10} b^2}{\epsilon_0^2}\left(s^{\star} \log(d) + \log \left(\frac{3}{\epsilon}\right) \right).
\end{align*}

Taking $n_0 = \max \left\{ 2^{12} b^2\epsilon^{-2} \left(s^{\star} \log(d) + \log \left(\frac{3}{\epsilon}\right) \right),  3200  s^{\star 2} \theta \log(d)\rho_0^{-2} \delta^{-2}, t\right\}$ completes the proof.
\end{proof}

\section{Producing a Lipschitz estimator}\label{section:Lipschitz}
We now modify Algorithm \ref{alg:integer-program} to ensure the Lipschitz property, in addition to the coordinate-wise monotone property. When we only needed to ensure monotonicity, the interpolation step was straightforward; interpolation was possible as long as the points themselves satisfied the monotonicity property. The situation is slightly more complicated for Lipschitz functions. Algorithm \ref{alg:integer-program-Lipschitz} ensures that the estimated points are \emph{interpolable} with respect to the class $\mathcal{L}_1(b)$, as defined below.
\begin{definition}
We say that a collection of points $(x_i, y_i)_{i=1}^n \in \mathbb{R}^k \times \mathbb{R}$ is \emph{interpolable} with respect to a function class $\mathcal{F}$ if there exists $f \in \mathcal{F}$ such that $f(x_i) = y_i$ for each $i \in [n]$.
\end{definition}

Algorithm \ref{alg:integer-program-Lipschitz} below finds the optimal index set and function values on a given set of points, compatible with interpolability. 
Binary variables $v_l$ determine the index set $I$. The variables $F_i$ represent the estimated function values at data points $X_i$. Auxiliary variables $z_{ij}$ and $w_{ijp}$ are used to model the monotonicity and Lipschitz constraints. \begin{breakablealgorithm}
\caption{Integer Programming Sparse Matrix Isotonic Regression (Lipschitz)}\label{alg:integer-program-Lipschitz}
\begin{algorithmic}[1]
\Require{Values $(X_i, Y_i)_{i=1}^n \in \mathbb{R}^d \times \mathbb{R}$, sparsity level $s$, $M \geq 0 \in \mathbb{R}^{d \times k}$, $C > 0$, $b > 0$}
\Ensure{An index set $I \subset [d]$ satisfying $|I| = s$; values $F_1, F_2, \dots, F_n \in [0,b]$ such that the points $(M(I)^T X_i, F_i)_{i=1}^n$ are interpolable by a coordinate-wise monotone $1$-Lipschitz function.}
\State Let $B = 2C \sum_{l=1}^d \sum_{p=1}^k M_{lp}$. 
\State Solve the following optimization problem.
\begin{align}
&\min_{v, F, z} \sum_{i=1}^n \left(Y_i - F_i\right)^2 \label{eq:objective-L}\\
\text{s.t. } &\sum_{l=1}^d v_l = s  \label{eq:sparsity-L}\\
&bz_{ij} \geq F_i - F_j & \forall i, j \in [n] \label{eq:ordering-L}\\
&(F_i - F_j)^2 \leq \sum_{p=1}^k w_{ijp} \left(\sum_{l=1}^d v_l M_{lp} (X_{il} - X_{jl})\right)^2 + (1 - z_{ij}) b^2 &\forall i, j \in [n] \label{eq:main-constraints-L}\\
&-B(1-w_{ijp}) \leq \sum_{l=1}^d v_l M_{lp} (X_{il} - X_{jl}) \leq B w_{ijp} &\forall i, j \in [n], p \in [k] \label{eq:w-constraints-L}\\
&z_{ij} \in \{0,1\} & \forall i, j \in [n] \nonumber\\
&v_l \in \{0,1\} &\forall l \in [d] \nonumber \\
&w_{ijp} \in \{0,1\} & \forall i, j \in [n], p \in [k] \nonumber\\
&F_i \in [0,b] &\forall i \in [n]  \nonumber
\end{align}
\State Return the set $I_n = \{ l \in [d] : v_l = 1\}$ and the values $F_1, F_2, \dots, F_n$.
\end{algorithmic}
\end{breakablealgorithm}
\begin{remark}
Note that Constraints \eqref{eq:main-constraints-L} contain products of three binary variables. We may encode arbitrary products of binary variables using linear constraints, as follows. Suppose $x$ and $y$ are binary variables, and we wish to encode $z = xy$. This is equivalent to the constraints $x + y - 1 \leq z \leq \frac{1}{2}(x+y)$ and $z \in \{0,1\}$. Longer products may be encoded recursively.
\end{remark}

We apply the construction of \cite{Beliakov2005} to find a coordinate-wise monotone, $1$-Lipschitz interpolation. 
\begin{proposition}\label{proposition:interpolation}
Suppose the points $(x_i, y_i)_{i=1}^n \in \mathbb{R}^k \times [0, b]$ satisfy 
\begin{align}
y_i - y_j \leq \Vert(x_i - x_j)^+ \Vert_2 \label{eq:criterion-prop}
\end{align}
for each pair $(i,j) \in [n]^2$. Let $\hat{g}(x) = \max_i \{y_i - \Vert (x_i - x)^+ \Vert_2 \}$, and let $\hat{f}(x) = \max \{\hat{g}(x), 0\}$. Then $\hat{f} \in\mathcal{L}_1(b)$. Furthermore, $\hat{f}$ interpolates the points; i.e. $\hat{f}(x_i) = y_i$ for each $i \in [n]$.
\end{proposition}
The proof follows from \cite{Beliakov2005}. We therefore obtain the following approach for interpolation.
\begin{breakablealgorithm}
\caption{Monotone Lipschitz Interpolation}\label{alg:interpolation}
\begin{algorithmic}[1]
\Require{Points $(x_i, y_i)_{i=1}^n \in \mathbb{R}^k \times [0, b]$ satisfying \eqref{eq:criterion-prop} for each $i,j$.}
\Ensure{An estimated function $\hat{f} \in \mathcal{L}_1(b)$ that interpolates the points.}
\State Let $\hat{g}(x) = \max_i \{y_i - \Vert (x_i - x)^+ \Vert_2 \}$. Return $\hat{f}(x) = \max \{\hat{g}(x), 0\}$.
% and let $\hat{f}_n(M(I_n)^Tx) = \max \{F_i : M(I_n)^TX_i \preceq M(I)^Tx\}$. Return $(\hat{f}_n, I_n)$.
%$\hat{f}_n(x) = \min_{1 \leq i \leq n} \{F_i : X_i \preceq x \}$.
\end{algorithmic}
\end{breakablealgorithm}

\begin{breakablealgorithm}
\caption{}\label{alg:Lipschitz-estimation}
\begin{algorithmic}[1]
\Require{Values $(X_i, Y_i)_{i=1}^n \in \mathbb{R}^d \times \mathbb{R}$, sparsity level $s$, $M \geq 0 \in \mathbb{R}^{d \times k}$, $C > 0$, $b > 0$}
\Ensure{$I_n \in [d] : |I_n| = s^{\star}$ and $f_n \in \mathcal{L}_1(b)$}
%\State Let $\tau = \left( \frac{3 Mn}{2 \log d} \right)^{\frac{1}{6}}$ and $\lambda = 10 \sqrt{M \frac{\log d}{n}}$,
\State Apply Algorithm \ref{alg:integer-program-Lipschitz} to input $(X_i, Y_i)_{i=n+1}^{2n}$, $s^{\star}$, $M$, $C$, and $b$, obtaining the index set $I$ and values $F_1, \dots, F_n$. 
\State Apply Algorithm \ref{alg:interpolation} to input $(M (I)X_{n+i}, F_i)_{i=1}^{n}$, obtaining the function $\hat{f}$.
\State Return $(I, \hat{f})$.
\end{algorithmic}
\end{breakablealgorithm}

\begin{proposition}\label{proposition:optimal-Lipschitz}
Suppose $X_i \in [-C, C]^d$ for $i \in [n]$. On input $(X_i, Y_i)_{i=1}^n, s, M, C, b$, Algorithm \ref{alg:Lipschitz-estimation} finds a function $\hat{f}_n \in \mathcal{L}_1(b)$ and index set $I_n$ that minimize the empirical loss $\sum_{i=1}^n L(X_i, Y_i, f \circ M(I))$, over functions $f \in \mathcal{L}_1(b)$ and index sets $I$ with cardinality $s$. 
\end{proposition}

We now modify Algorithm \ref{alg:full} to include the Lipschitz assumption.
\begin{breakablealgorithm}
\caption{MMI Regression (Lipschitz)}\label{alg:full-Lipschitz}
\begin{algorithmic}[1]
\Require{$N_0 \in \mathbb{N}$, values $(X_1, Y_1), \dots, (X_N, Y_N)$, $C > 0$, $b > 0$, $\tau > 0$, and $\lambda > 0$}
\Ensure{$f_n \in \mathcal{L}_1(b)$, $Q_n \in \mathcal{O}_{d,k}$, $\overline{R}_n \in \overline{\mathcal{M}}_{k,k}(r)$, and $I_n \in [d] : |I_n| = s^{\star}$}
\State Construct a random near-net $\mathcal{R}(N_0)$.
%\State Let $\tau = \left( \frac{3 Mn}{2 \log d} \right)^{\frac{1}{6}}$ and $\lambda = 10 \sqrt{M \frac{\log d}{n}}$,
\State Produce an estimate $Q_n$ using Algorithm \ref{alg:SDP} applied to $(X_i, Y_i)_{i=1}^n$, $\tau$, and $\lambda$.
\ForEach {$R \in \mathcal{R}$} 
%\State Apply Algorithm \ref{alg:integer-program-Lipschitz} to input $(X_i, Y_i)_{i=n+1}^{2n}$, $s^{\star}$, $(Q_nR)^+$, $C$, and $b$, obtaining the index set $I_R$ and values $F_1, \dots, F_n$. 
\State Let $M= (Q_nR)^+$. Apply Algorithm \ref{alg:Lipschitz-estimation} to input $(X_i, Y_i)_{i=n+1}^n \in \mathbb{R}^d \times \mathbb{R}$, $s^{\star}$, $M$, $C$, and $b$, obtaining the index set $I_R$ and function $f_R$.
\EndFor
\State Return the tuple $(f_R, Q_n, R, I_R)$ with the smallest empirical loss.
\end{algorithmic}
\end{breakablealgorithm}
The proof of Theorem \ref{theorem:main-result} carries through exactly for Algorithm \ref{alg:full-Lipschitz}, since $\mathcal{L}_1(b) \subset \mathcal{C}(b)$. We note that a tighter analysis of the estimation error incurred by using Algorithm \ref{alg:full-Lipschitz} would take advantage of the Lipschitz property of the estimated function.

\subsection{Proofs}
In order to prove Proposition \ref{proposition:interpolation}, we need to know when a collection of points is interpolable by a coordinate-wise monotone and $1$-Lipschitz function. The following result provides a necessary and sufficient condition for interpolability.
\begin{proposition}[From Proposition 3.3 and Proposition 4.1 in \cite{Beliakov2005}]\label{proposition:criterion}
A collection of points $(x_i, y_i)_{i=1}^n \in \mathbb{R}^k \times \mathbb{R}$ 
is interpolable with respect to the class of coordinate-wise monotone and $1$-Lipschitz functions if and only if
\begin{align*}
y_i - y_j \leq \Vert (x_i - x_j)^+ \Vert_2 
\end{align*}
for all $i, j \in [n]$. Further, if the collection is interpolable, then the function
\[\hat{f}(x) = \max_i \{y_i - \Vert (x_i - x)^+ \Vert_2 \]
is an interpolation that is coordinate-wise monotone and $1$-Lipschitz.
\end{proposition}

\begin{proof}[Proof of Proposition \ref{proposition:interpolation}]
By Proposition \ref{proposition:criterion},
% a data set $(x_1, y_i)_{i=1}^n$ admits a $1$-Lipschitz coordinate-wise monotone interpolation if and only if
%\[y_i - y_j \leq \Vert (x_i - x_j)^+ \Vert_2 \]
%for each $i, j \in [n]$. By the assumption, 
the data admits an interpolation by a coordinate-wise monotone $1$-Lipschitz function. Further, the function $\hat{f}$ interpolates the data, and is coordinate-wise monotone and $1$-Lipschitz. Since $y_i \geq 0$ for all $i$, the function $\hat{g}$ interpolates the data also. The zero function is $1$-Lipschitz and coordinate-wise monotone. Therefore, $\hat{g}$, which is the pointwise maximum of two $1$-Lipschitz and coordinate-wise monotone functions, is itself $1$-Lipschitz and coordinate-wise monotone.

It remains to show that $0 \leq \hat{g}(x) \leq b$ for all $x$. Clearly $\hat{g}(x) \geq 0$ for all $x$. Since $y_i \leq b$ for all $i$, it holds that 
\[\hat{f}(x) \leq \max_i \{b - \Vert (x_i - x)^+ \Vert_2\} \leq b \implies \hat{g}(x) \leq b.\]
\end{proof}

\begin{proof}[Proof of Proposition \ref{proposition:optimal-Lipschitz}]
First we show that Algorithm \ref{alg:integer-program-Lipschitz} finds an index set $I$ and values $F_1, \dots, F_n \in [0,b]$ minimizing $\sum_{i=1}^n (Y_i - F_i)^2$, such that the points $(M(I)^TX_i, F_i)_{i=1}^n$ are interpolable by a coordinate-wise monotone $1$-Lipschitz function. Later, we will show that the points are in fact interpolable by a coordinate-wise monotone $1$-Lipschitz function with range $[0,b]$, a more restrictive requirement.

Let $I = \{i : v_i =1\}$. Constraint \eqref{eq:sparsity-L} ensures that exactly $s$ of the $v_l$ variables are set to $1$, so that $|I| = s$. Given this index set, we will show that the points $(M(I)^TX_i, F_i)_{i=1}^n$ are interpolable by a coordinate-wise monotone $1$-Lipschitz function. By Proposition \ref{proposition:criterion}, this is equivalent to 
\[F_i - F_j \leq \Vert \left(M(I)^T X_i - M(I)^T X_j\right)^+ \Vert_2,\]
for each $i,j$. First observe that this is equivalent to either (a) $F_i \leq F_j$ or \[\text{(b)~~}(F_i - F_j)^2 \leq \Vert \left(M(I)^T X_i - M(I)^T X_j\right)^+ \Vert_2^2.\]
We now show how the constraints encode this condition. First suppose $F_i > F_j$. Then by Constraint \eqref{eq:ordering-L}, $z_{ij} = 1$, and
\begin{align*}
&F_i - F_j \leq \Vert \left(M(I)^T X_i - M(I)^T X_j\right)^+ \Vert_2 \iff (F_i - F_j)^2 \leq \Vert \left(M(I)^T X_i - M(I)^T X_j\right)^+ \Vert_2^2\\
&\iff (F_i - F_j)^2 \leq \Vert \left(M(I)^T X_i - M(I)^T X_j\right)^+ \Vert_2^2 + (1-z_{ij})b^2.
\end{align*}
Next suppose $F_i \leq F_j$. Then $z_{ij}$ is free to equal $0$, so that 
\[(F_i - F_j)^2 \leq \Vert \left(M(I)^T X_i - M(I)^T X_j\right)^+ \Vert_2^2 + (1-z_{ij})b^2.\]
We conclude that the points $(M(I)^TX_i, F_i)_{i=1}^n$ are interpolable by a coordinate-wise monotone $1$-Lipschitz function if and only if 
\begin{align}
(F_i - F_j)^2 \leq \Vert \left(M(I)^T X_i - M(I)^T X_j\right)^+ \Vert_2^2 + (1-z_{ij})b^2. \label{eq:condition}
\end{align}
for each pair $(i,j)$, where $z_{ij} = 1$ if $F_i > F_j$. Expanding the right hand side of \eqref{eq:condition},
\begin{align*}
\Vert \left(M(I)^T X_i - M(I)^T X_j\right)^+ \Vert_2^2 &= \sum_{p=1}^k \left( \sum_{l=1}^d (M(I)^T)_{pl} (X_{il} - X_{jl})\right)^{+,2}\\
&= \sum_{p=1}^k \left( \sum_{l=1}^d v_l M_{lp} (X_{il} - X_{jl})\right)^{+,2}.
\end{align*}
Now, Constraints \eqref{eq:w-constraints-L} ensure that 
\begin{align*}
w_{ijp} &= \begin{cases}
1 & \text{if }  \sum_{l=1}^d v_l M_{lp} (X_{il} - X_{jl}) > 0\\
0 & \text{if }  \sum_{l=1}^d v_l M_{lp} (X_{il} - X_{jl}) < 0.
\end{cases}
\end{align*}

Therefore,
\begin{align*}
\Vert \left(M(I)^T X_i - M(I)^T X_j\right)^+ \Vert_2^2 &= \sum_{p=1}^k w_{ijp} \left( \sum_{l=1}^d v_l M_{lp} (X_{il} - X_{jl})\right)^{2}.
\end{align*}
We conclude that the points $(M(I)^TX_i, F_i)_{i=1}^n$ are interpolable by a coordinate-wise monotone $1$-Lipschitz function if and only if 
\begin{align*}
(F_i - F_j)^2 \leq \sum_{p=1}^k w_{ijp} \left( \sum_{l=1}^d v_l M_{lp} (X_{il} - X_{jl})\right)^2 + (1-z_{ij})b^2,
\end{align*}
for each $(i,j)$, which is exactly Constraints \eqref{eq:main-constraints-L}.

The objective \eqref{eq:objective-L} minimizes the loss on the samples. Finally, the interpolation step produces a coordinate-wise monotone Lipschitz function with range $[0,b]$, by Proposition \ref{proposition:interpolation}.
 \end{proof}

\end{document}